%% file: 15_Chandler-Wilde_Hewett_WavenumberExplicit.tex
\documentclass{birkjour}

\usepackage{amsmath}
\usepackage{mathtools}
\usepackage{stmaryrd}
\usepackage{amsfonts}
\usepackage[mathscr]{eucal}
\usepackage{amssymb}
\usepackage{enumerate}
\usepackage{cases}
\usepackage{multirow}
\usepackage{subfigure}
\usepackage{graphicx}
\usepackage{appendix}
\usepackage{pgf,tikz}
\usepackage{color}
\usepackage[show]{ed} %
\DeclareMathOperator{\sign}{sgn}
\DeclareMathOperator{\supp}{supp}

\DeclareMathOperator{\diam}{diam}
\DeclareMathOperator{\dist}{dist}

\begin{document}
\input{macros}
\newcommand{\tk}{k}
\newcommand{\tbx}{\tilde{\bx}}
\newcommand{\tby}{\tilde{\by}}	%
\newcommand{\tbd}{\tilde{\bd}}
\newcommand{\txi}{\xi}
\newcommand{\bxi}{\boldsymbol{\xi}}
\newcommand{\boldeta}{\boldsymbol{\eta}}	%
\newcommand{\sD}{\mathsf{D}}
\newcommand{\sN}{\mathsf{N}}
\newcommand{\sE}{\mathsf{E}}
\newcommand{\sS}{\mathsf{S}}
\newcommand{\sT}{\mathsf{T}}
\newcommand{\sH}{\mathsf{H}}
\newcommand{\sI}{\mathsf{I}}
\renewcommand{\cA}{A}
\newcommand{\cB}{\mathcal{B}}
\newcommand{\cH}{\mathcal{H}}
\newcommand{\cI}{\mathcal{I}}
\newcommand{\cItilde}{\tilde{\mathcal{I}}}
\newcommand{\cIhat}{\hat{\mathcal{I}}}
\newcommand{\cIcheck}{\check{\mathcal{I}}}
\newcommand{\cIstar}{{\mathcal{I}^*}}
\newcommand{\cJ}{\mathcal{J}}
\newcommand{\cM}{\mathcal{M}}
\newcommand{\cT}{\mathcal{T}}
\newcommand{\scrD}{\mathscr{D}}
\newcommand{\scrS}{\mathscr{S}}
\newcommand{\scrJ}{\mathscr{J}}
\newcommand{\tH}{\tilde{H}}
\newcommand{\Hze}{H_{\rm ze}}
\newcommand{\dimH}{{\rm dim_H}}
\newcommand{\dudnjump}{\left[ \pdone{u}{\bn}\right]}
\newcommand{\dudnjumptext}{[ \pdonetext{u}{\bn}]}
\newcommand{\sumpm}[1]{\{\!\!\{#1\}\!\!\}}
\newtheorem{example}[thm]{Example}
\title[Acoustic scattering by planar screens]{Wavenumber-explicit continuity and coercivity estimates in acoustic scattering by planar screens}

\author{S.\ N.\ Chandler-Wilde}
\address{Department of Mathematics and Statistics\\ University of Reading \\ 
UK.}
\email{s.n.chandler-wilde@reading.ac.uk}
\author{D.\ P.\ Hewett}
\address{Department of Mathematics and Statistics\\ University of Reading \\ 
UK. \\
Current address:\\ Mathematical Institute\\ University of Oxford\\ %
UK.}
\email{hewett@maths.ox.ac.uk}
\begin{abstract}
We study the classical first-kind boundary integral equation reformulations of time-harmonic acoustic scattering by planar sound-soft (Dirichlet) and sound-hard (Neumann) screens. We prove continuity and coercivity of the relevant boundary integral operators (the acoustic single-layer and hypersingular operators respectively) in appropriate fractional Sobolev spaces, with wavenumber-explicit bounds on the continuity and coercivity constants.  Our analysis, which requires no regularity assumptions on the boundary of the screen (other than that the screen is a relatively open bounded subset of the plane), is based on spectral representations for the boundary integral operators, and builds on results of Ha-Duong (Jpn J Ind Appl Math 7:489--513 (1990) and Integr Equat Oper Th 15:427--453 (1992)). %
\end{abstract}
\subjclass{65R20, 35Q60}
\keywords{Boundary integral equations, Wave scattering, Screen problems, Single-layer operator, Hypersingular operator}
\maketitle
\section{Introduction}
\label{sec:Intro}
This paper concerns the mathematical analysis of a class of time-harmonic acoustic scattering problems modelled by the Helmholtz equation %
\begin{equation} \label {eq:he}
\Delta u + k^2 u = 0,
\end{equation}
where $u$ is a complex scalar function and $k>0$ is the {\em wavenumber}.
We study the reformulation of such scattering problems in terms of boundary integral equations (BIEs), proving continuity and coercivity estimates for the associated boundary integral operators (BIOs) which are explicit in their $k$-dependence.

Our focus is on scattering by a thin planar screen occupying some bounded and relatively open set $\Gamma \subset \Gamma_\infty:=\{\bx=(x_1,...,x_n)\in \R^n:x_n=0\}$ (we assume throughout
that $n=2$ or $3$), with \eqref{eq:he} assumed to hold in the domain $D:=\R^n\setminus \overline{\Gamma}$. 
We consider both the Dirichlet (sound-soft) and Neumann (sound-hard) boundary value problems (BVPs), which we now state.
The function space notation in the following definitions, and the precise sense in which the boundary conditions are to be understood, is explained in \S\ref{sec:prelims}.

\begin{defn}[Problem $\sD$]
\label{def:SPD}
Given $g_{\sD}\in H^{1/2}(\Gamma)$, find $u\in C^2\left(D\right)\cap  W^1_{\mathrm{loc}}(D)$ such that
\vs{-2}
\begin{align}
 \Delta u+k^2u  &=  0, \;\;\quad\mbox{in }D,  \label{eqn:HE1} \\
u  &= g_{\sD}, \quad\mbox{on }\Gamma,\label{eqn:bc1}
\end{align}
and $u$ satisfies the Sommerfeld radiation condition at infinity. %
\end{defn}
\begin{defn}[Problem $\sN$]
\label{def:SPN}
Given $g_{\sN}\in H^{-1/2}(\Gamma)$, find $u\in C^2\left(D\right)\cap W^1_{\mathrm{loc}}(D)$ such that
\vs{-2}
\begin{align}
  \Delta u+k^2u  &=  0, \;\;\quad \mbox{in }D,  \label{eqn:HE2} \\
 \pdone{u}{\bn}  &= g_{\sN}, \quad \mbox{on }\Gamma, \label{eqn:bc2}%
\end{align}
and $u$ satisfies the Sommerfeld radiation condition at infinity. %
\end{defn}

\begin{example}
\label{ex:scattering}
Consider the problem of scattering by $\Gamma$ of an incident
plane wave
\begin{align}
\label{uiDef}
  u^i(\bx)&:=\re^{\ri  k \bx\cdot \bd}, \qquad \bx\in\mathbb{R}^n,%
\end{align}
where
$\bd\in\mathbb{R}^n$ is a unit direction vector.
A `sound-soft' and a `sound-hard' screen are modelled respectively by problem $\sD$ (with $g_{\sD}=-u^i|_\Gamma$) and problem $\sN$ (with $g_{\sN}=-\pdonetext{u^i}{\bn}|_\Gamma$). In both cases $u$ represents the scattered field, the total field being given by $u^i+u$.
\end{example}

Such scattering problems have been well-studied, both theoretically \cite{StWe84,Ste:87,WeSt:90,Ha-Du:90,Ha-Du:92,HoMaSt:96} and in applications \cite{DLaWroPowMan:98,DavisChew08}. It is well known (see, e.g., \cite{StWe84,Ste:87,WeSt:90,HoMaSt:96}) that problems $\sD$ and $\sN$ are uniquely solvable for all $g_{\sD}\in H^{1/2}(\Gamma)$ and $g_{\sN}\in H^{-1/2}(\Gamma)$, provided that $\Gamma\subset \Gamma_\infty$ is sufficiently smooth. Many of the references cited above assume $\Gamma\subset \Gamma_\infty$ is $C^\infty$ smooth, or do not explicitly specify the regularity of $\Gamma$, but, as has been clarified recently in \cite{CoercScreen}, unique solvability holds whenever $\Gamma$ is Lipschitz (in the sense considered e.g.\ in \cite[p.~90]{McLean}). The case of non-Lipschitz $\Gamma$ can also be considered (for details see \cite{CoercScreen}), but in general requires problems $\sD$ and $\sN$ to be supplemented with additional assumptions in order to guarantee uniqueness. For simplicity of presentation we do not consider such generalisations here, and restrict our discussion of the BVPs $\sD$ and $\sN$ to the case of Lipschitz $\Gamma$. However, our results on the associated BIOs (which are the main focus of this paper) are valid for arbitrary relatively open $\Gamma$, as will be clarified at the start of \S\ref{subsec:Outline}.

The solutions of problems $\sD$ and $\sN$ can be represented respectively in terms of the single and double layer potentials (for notation and definitions see \S\ref{sec:prelims})
\begin{align*}
\label{}
\cS_k:\tH^{-1/2}(\Gamma)\to C^2(D)\cap W^1_{\rm loc}(D),\qquad \cD_k:\tH^{1/2}(\Gamma)\to C^2(D)\cap W^1_{\rm loc}(D),
\end{align*}
which for
$\phi\in\scrD(\Gamma)$
and $\bx\in D$
have the following integral representations:%
\begin{align}
\label{SkPotIntRep}
\cS_k\phi (\bx) = \int_\Gamma \Phi(\bx,\by) \phi(\by)\, \rd s(\by), %
\qquad
\cD_k\phi (\bx) = \int_\Gamma \pdone{\Phi(\bx,\by)}{\bn(\by)} \phi(\by)\, \rd s(\by), %
\end{align}
where $\Phi(\bx,\by)$ denotes the fundamental solution of \rf{eq:he},
\begin{align}
\label{FundSoln}
\Phi(\bx,\by) :=
\begin{cases}
\dfrac{\re^{\ri k|\bx-\by|}}{4\pi |\bx-\by|}, & n=3,\\[3mm]
\dfrac{\ri }{4}H_0^{(1)}(k|\bx-\by|), & n=2,
\end{cases}
\qquad \bx,\by \in \R^n.
\end{align}
The densities of the potentials satisfy certain first-kind BIEs involving the single-layer and hypersingular BIOs (again, for definitions see \S\ref{sec:prelims})
\begin{align*}
\label{}
S_k&:\tH^{-1/2}(\Gamma)\to H^{1/2}(\Gamma)\cong(\tH^{-1/2}(\Gamma))^*,\\ T_k&:\tH^{1/2}(\Gamma)\to H^{-1/2}(\Gamma)\cong (\tH^{1/2}(\Gamma))^*,
\end{align*}
which for $\phi\in \scrD(\Gamma)$ and $\bx\in \Gamma$ have the following integral representations:
\begin{align}
\label{SDef}
S_k\phi (\bx) = \int_\Gamma \Phi(\bx,\by) \phi(\by)\, \rd s(\by), %
\qquad
T_k\phi (\bx) = \pdone{}{\bn(\bx)}\int_\Gamma \pdone{\Phi(\bx,\by)}{\bn(\by)} \phi(\by)\, \rd s(\by).%
\end{align}
These standard statements are summarised in the following two theorems. Here $[u]$ and $[\pdonetext{u}{\bn}]$ represent the jump across $\Gamma$ of $u$ and of its normal derivative respectively. 

\begin{thm}
\label{DirEquivThm}
For Lipschitz $\Gamma$,
problem $\sD$ has a unique solution $u$ satisfying %
\begin{align}
\label{eqn:SLPRep}
u(\bx )= -\cS_k\left[\pdonetext{u}{\bn}\right](\bx), \qquad\bx\in D,
\end{align}
where $[\pdonetext{u}{\bn}]\in \tH^{-1/2}(\Gamma)$ is the unique solution of the BIE
\begin{align}
\label{BIE_sl}
-S_k[\pdonetext{u}{\bn}]= g_{\sD}.
\end{align}
\end{thm}
\begin{thm}
\label{NeuEquivThm}
For Lipschitz $\Gamma$, 
problem $\sN$ has a unique solution satisfying %
\begin{align}
\label{eqn:DLPRep}
 u(\bx ) =  \cD_k[u](\bx), \qquad\bx\in D,
\end{align}
where $[u]\in \tH^{1/2}(\Gamma)$ is the unique solution of the BIE
\begin{align}
\label{BIE_hyp}
T_k[u] = g_{\sN}.
\end{align}
\end{thm}

\subsection{Main results and outline of the paper}
\label{subsec:Outline}
In this paper we present new $k$-explicit continuity and coercivity estimates for the operators $S_k$ and $T_k$ appearing in \rf{BIE_sl} and \rf{BIE_hyp}. Our main results are contained in Theorems \ref{ThmSNormSmooth}-\ref{ThmTCoercive} below. 
We emphasize that Theorems \ref{ThmSNormSmooth}-\ref{ThmTCoercive} hold for any relatively open subset $\Gamma\subset\Gamma_\infty$, without any regularity assumption on $\Gamma$. (The Lipschitz regularity assumption in Theorems \ref{DirEquivThm}-\ref{NeuEquivThm} ensures equivalence between the relevant BVPs and the BIEs, but is not necessary for continuity and coercivity of the BIOs.) 
For the definitions of the $k$-dependent norms appearing in Theorems \ref{ThmSNormSmooth}-\ref{ThmTCoercive} see \S\ref{sec:SobolevSpaces}. 

\begin{thm}
\label{ThmSNormSmooth}
For any $s\in \R$, the single-layer operator $S_k$ defines a bounded linear operator $S_k:\tH^{s}(\Gamma)\to H^{s+1}(\Gamma)$, and there exists a constant $C>0$, independent of $k$ and $\Gamma$, such that, for all $\phi\in \tH^{s}(\Gamma)$, and with $L:=\diam{\Gamma}$,
\begin{align}
\label{SEst}
\norm{S_k \phi}{H^{s+1}_k(\Gamma)} \leq
\begin{cases}
C(1+ (kL)^{1/2})\norm{\phi}{\tilde{H}^{s}_k(\Gamma)}, & n=3,\\[2mm]
C \log{(2+(kL)^{-1})}(1+ (kL)^{1/2})\norm{\phi}{\tilde{H}^{s}_k(\Gamma)}, & n=2,
\end{cases}
\quad k>0.
\end{align}
\end{thm}

\begin{thm}
\label{ThmCoerc}
The sesquilinear form on $\tilde{H}^{-1/2}(\Gamma)\times \tilde{H}^{-1/2}(\Gamma)$ defined by
\begin{align*}
\label{}
a_\sD(\phi,\psi):=\langle S_k\phi,\psi \rangle_{H^{1/2}(\Gamma)\times \tilde{H}^{-1/2}(\Gamma)}, \quad \phi,\psi\in \tilde{H}^{-1/2}(\Gamma),
\end{align*}
satisfies the coercivity estimate
\begin{align}
\label{SCoercive}
|a_\sD(\phi,\phi)|
\geq \frac{1}{2\sqrt{2}} \norm{\phi}{\tilde{H}^{-1/2}_k(\Gamma)}^2,\quad \phi\in \tilde{H}^{-1/2}(\Gamma), \,\, k>0.
\end{align}
\end{thm}

\begin{thm}
\label{ThmTNormSmooth}
For any $s\in\R$, the hypersingular operator $T_k$ defines a bounded linear operator $T_k:\tH^{s}(\Gamma)\to H^{s-1}(\Gamma)$, and
\begin{align}
\label{TEst}
\norm{T_k \phi}{H^{s-1}_k(\Gamma)} \leq \frac{1}{2}\norm{\phi}{\tilde{H}^{s}_k(\Gamma)}, \quad \phi\in \tH^{s}(\Gamma), \,\,k>0.
\end{align}
\end{thm}

\begin{thm}
\label{ThmTCoercive}
The sesquilinear form on $\tilde{H}^{1/2}(\Gamma)\times \tilde{H}^{1/2}(\Gamma)$ defined by
\begin{align*}
\label{}
a_\sN(\phi,\psi):=\langle T_k\phi,\psi \rangle_{H^{-1/2}(\Gamma)\times \tilde{H}^{1/2}(\Gamma)}, \quad \phi,\psi\in \tilde{H}^{1/2}(\Gamma),
\end{align*}
satisfies, for any $c_0>0$, the coercivity estimate %
\begin{align}
\label{bEst}
|a_\sN(\phi,\phi)|\geq
C(kL)^{\beta}
\norm{\phi}{\tilde{H}^{1/2}_k(\Gamma)}^2, \quad \phi\in \tilde{H}^{1/2}(\Gamma),\,\,kL\geq c_0,
\end{align}
where $L:=\diam{\Gamma}$,
$C>0$ is a constant depending only on $c_0$, %
and
\begin{align}
\label{eq:beta}
\beta=
\begin{cases}
-\frac{2}{3},&n=3,\\[1mm]
-\frac{1}{2}, & n=2.
\end{cases}
\end{align}
\end{thm}

Our proofs of these theorems are given in \S\ref{sec:SkAnalysis} and \S\ref{sec:TkAnalysis} below. In Remarks \ref{rem:sharpness1}, \ref{SRemcoer}, \ref{TRem}, and \ref{rem:sharpness7}, we show that the estimates in Theorems \ref{ThmSNormSmooth}-\ref{ThmTCoercive} are sharp in their dependence on $k$ in the high frequency limit $k\to\infty$, with one exception: we suspect that Theorem \ref{ThmTCoercive} may be true in the case $n=3$ with the same value $\beta=-1/2$ as in the 2D case, this conjecture consistent with Remark \ref{TRem} below.

The implications of these results for the analysis of high frequency acoustic scattering problems will be discussed in \S\ref{sec:Motivation} below. But first we provide a brief overview of the structure of the rest of the paper and a comparison with related literature.

The technical definitions and notation for the function spaces, norms, trace and jump operators, layer potentials and BIOs that we study are presented in \S\ref{sec:prelims}. Our assumption that the screen $\Gamma$ is planar means that Sobolev spaces on $\Gamma$ can be defined concretely for all orders of Sobolev regularity in terms of Fourier transforms on the hyperplane $\Gamma_\infty$, which we naturally associate with $\R^{n-1}$.  %
Furthermore, the planarity of $\Gamma$ also allows us to derive explicit Fourier transform representations for the layer potentials and BIOs, which we present in \S\ref{sec:FourierRep}. These representations facilitate our wavenumber-explicit continuity and coercivity analysis, which is presented in \S\ref{sec:SkAnalysis} and \S\ref{sec:TkAnalysis}.
In this respect our wavenumber-explicit analysis extends that carried out using similar arguments by Ha-Duong in \cite{Ha-Du:90,Ha-Du:92}\footnote{
The authors are grateful to M.\ Costabel for drawing references \cite{Ha-Du:90,Ha-Du:92} to their attention.}.
Indeed, a wavenumber-explicit coercivity estimate for the hypersingular operator $T_k$ of the form \rf{bEst} is proved in \cite[Theorem 2]{Ha-Du:92}, but only for the case $n=3$ and with $\beta=1$ (although the methodology in \cite{Ha-Du:92} can be modified to deal with the case $n=2$, also giving $\beta=1$). We have been able to improve this to $\beta=2/3$ in the case $n=3$ and $\beta=1/2$ in the case $n=2$.
To the best of our knowledge, wavenumber-explicit coercivity estimates for the single-layer operator $S_k$ have not been published before; the fact that $S_k$ is coercive is stated without proof in \cite[Prop.\ 2.3]{Co:04}, with a reference to \cite{Ha-Du:90}, but in \cite{Ha-Du:90} coercivity is only proved for complex wavenumber, the real case being mentioned only in passing (see \cite[p.\ 502]{Ha-Du:90}).
Finally, in \S\ref{sec:NormEstimates} we collect some useful norm estimates in the space $H^{1/2}(\Gamma)$, of relevance in the numerical analysis of Galerkin boundary element methods (BEMs) based on the BIE reformulation \rf{BIE_sl} of the Dirichlet screen problem (see \cite{ScreenBEM} for an application of these results).

We remark that some of the results in this paper were stated without proof in the conference paper \cite{HeCh:13}.

\subsection{Motivation}
\label{sec:Motivation}
The wavenumber-explicit analysis presented in this paper forms part of a wider effort in the rigorous mathematical analysis of BIE methods for high frequency acoustic scattering problems (for a recent review of this area see e.g.\ \cite{ChGrLaSp:11}). Typically (and this is the approach we adopt here) one reformulates the scattering problem as an integral equation, which may be written in operator form as
\begin{equation} \label{eq:ieab}
\cA \phi = f,
\end{equation}
where $\phi$ and $f$ are complex-valued functions defined on the boundary $\Gamma$ of the scatterer. A standard and appropriate functional analysis framework
is that the solution $\phi$ is sought in some Hilbert space $V$, with $f\in V^*$, the dual space of $V$ (the space of continuous antilinear functionals), and $\cA:V\to V^*$ a bounded linear BIO.\footnote{A concrete example is the standard Brakhage-Werner formulation \cite{BraWer:65,ChGrLaSp:11} of sound-soft acoustic scattering by a bounded, Lipschitz obstacle, in which case $V=V^* = L^2(\Gamma)$, and
\begin{equation} \label{eq:ABW}
\cA = \frac{1}{2}I + D_k -\ri \eta S_k,
\end{equation}
with $S_k$ and $D_k$ the standard acoustic single- and double-layer BIOs, $I$ the identity operator, and $\eta\in \R\setminus\{0\}$ a coupling parameter.}
Equation \rf{eq:ieab} can be restated in weak (or variational) form as
\begin{align}
\label{WeakForm}
a(\phi, \varphi) = f(\varphi), \quad \textrm{for all } \varphi\in V,
\end{align}
in terms of the sesquilinear form
\begin{align*}
\label{}
a(\phi, \varphi):=(\cA \phi)(\varphi), \quad \phi,\varphi\in V.
\end{align*}
The Galerkin method for approximating \rf{WeakForm} seeks a solution $\phi_N\in V_N\subset V$, where $V_N$ is a finite-dimensional subspace, requiring that
\begin{align}
\label{WeakFormGalerkin}
a(\phi_N, \varphi_N) = f(\varphi_N), \quad \textrm{for all } \varphi_N\in V_N.
\end{align}

The sesquilinear form $a$ is clearly bounded with continuity constant equal to $\|\cA\|_{V\to V^*} $.
We say that $a$ (and the associated bounded linear operator $\cA$) is \emph{coercive} if, for some $\gamma>0$ (called the \emph{coercivity constant}), it holds that
\begin{align*}
\label{}
|a(\phi,\phi)|\geq \gamma\|\phi\|_{V}^2, \quad \textrm{for all } \phi\in V.
\end{align*}
In this case, the Lax-Milgram lemma implies that \rf{WeakForm} (and hence \rf{eq:ieab}) has exactly one solution $\phi\in V$, and that $\|\phi\|_{V} \leq \gamma^{-1} \|f\|_{V^*}$, i.e.\ $\|A^{-1}\|_{V^*\to V} \leq \gamma^{-1}$.
Furthermore, by C\'ea's lemma, the existence and uniqueness of the Galerkin solution $\phi_N$ of \rf{WeakFormGalerkin} is then also guaranteed for any finite-dimensional approximation space $V_N$, and there holds the quasi-optimality estimate
\begin{align}
\label{QuasiOpt}
\|\phi -\phi_N\|_{V} \leq \frac{\|\cA\|_{V\to V^*}}{\gamma} \inf_{\varphi_N\in V_N} \| \phi- \varphi_N\|_{V}.
\end{align}

One major thrust of recent work (for a review see \cite{ChGrLaSp:11}) has been to attempt to prove wavenumber-explicit continuity and coercivity estimates for BIE formulations of scattering problems, which, by the above discussion, lead to wavenumber-explicit bounds on the condition number $\|A\|_{V\to V^*}\|A^{-1}\|_{V^*\to V}$ and the quasi-optimality constant $\gamma^{-1}\|\cA\|_{V\to V^*}$.
\footnote{Such estimates have recently been proved \cite{SpSm:11} for the operator \rf{eq:ABW} for the case where the scatterer is strictly convex and $\Gamma$ is sufficiently smooth. We also note that in \cite{SCWGS11} a new formulation for sound-soft acoustic scattering, the so-called `star-combined' formulation, has been shown to be coercive on $L^2(\Gamma)$ for all star-like Lipschitz scatterers.}
This effort is motivated by the fact that these problems are computationally challenging when the wavenumber $k>0$ (proportional to the frequency) is large, and that such wavenumber-explicit estimates are useful for answering certain key numerical analysis questions, for instance:
\begin{enumerate}[(a)]
\item Understanding the behaviour of iterative solvers (combined with matrix compression techniques such as the fast multipole method) at high frequencies, in particular understanding the dependence of iteration counts on parameters related to the wavenumber. This behaviour depends, to a crude first approximation, on the condition number of the associated matrices, which is in part related to the wavenumber dependence of the norms of the BIOs and their inverses at the continous level (\cite{BeChGrLaLi:11,SpSm:11}).
\item Understanding the accuracy of conventional BEMs (based on piecewise polynomial approxiomation spaces) at high frequencies by undertaking a rigorous numerical analysis which teases out the joint dependence of the error on the number of degrees of freedom and the wavenumber $k$. For example, is it enough to increase the degrees of freedom in proportion to $k^{d-1}$ in order to maintain accuracy, maintaining a fixed number of degrees of freedom per wavelength in each coordinate direction? See e.g.\ \cite{LohMel:11,GrLoMeSp:13} for some recent results in this area.
\item Developing, and justifying by a complete numerical analysis, novel BEMs for high frequency scattering problems based on the so-called `hybrid numerical-asymptotic' (HNA) approach, the idea of which is to use an approximation space enriched with oscillatory basis functions, carefully chosen to capture the high frequency solution behaviour. The aim is to develop algorithms for which the number of degrees of freedom $N$ required to achieve any desired accuracy be fixed or increase only very mildly as $k\to\infty$. This aim is provably achieved in certain cases, mainly 2D so far; see, e.g., \cite{DHSLMM,NonConvex} and the recent review \cite{ChGrLaSp:11}. For 2D screen and aperture problems we recently proposed in \cite{ScreenBEM} an HNA BEM which provably achieves a fixed accuracy of approximation with $N$ growing at worst like $\log^2{k}$ as $k\to\infty$, our numerical analysis using the wavenumber-explicit estimates of the current paper.
Numerical experiments demonstrating the effectiveness of HNA approximation spaces for a 3D screen problem have been presented in \cite[\S7.6]{ChGrLaSp:11}.
\end{enumerate}

Clearly the results in this paper are a contribution to this endeavour. In particular, Theorems \ref{ThmSNormSmooth} and \ref{ThmTNormSmooth} provide upper bounds on
$
\|S_k\|_{\tilde H^{-1/2}(\Gamma)\to H^{1/2}(\Gamma)}$ and $\|T_k\|_{\tilde H^{1/2}(\Gamma)\to H^{-1/2}(\Gamma)}$ (with $\tilde H^{\pm 1/2}(\Gamma)$ and $H^{\pm 1/2}(\Gamma)$ equipped with the wavenumber-dependent norms specified in \S\ref{sec:prelims}).
Further, as noted generically above, through providing lower bounds on the coercivity constant $\gamma$, Theorems \ref{ThmCoerc} and \ref{ThmTCoercive} bound the inverses of these operators. Thus our results also provide bounds on condition numbers: in particular, for every $c_0>0$, our results show that, for $kL\geq c_0$,
\begin{align} \label{eq:cond}
\mathrm{cond}\,S_k := \|S_k\|_{\tilde H^{-1/2}(\Gamma)\to H^{1/2}(\Gamma)}\|S_k\|_{H^{1/2}(\Gamma)\to \tilde H^{-1/2}(\Gamma)} \leq C (kL)^{1/2},\\
\mathrm{cond}\,T_k := \|T_k\|_{\tilde H^{1/2}(\Gamma)\to H^{-1/2}(\Gamma)}\|T_k\|_{H^{-1/2}(\Gamma)\to \tilde H^{1/2}(\Gamma)} \leq C (kL)^{-\beta},
\end{align}
where $\beta$ is given by \eqref{eq:beta} and $C>0$ is a constant that depends only on $c_0$. Further, Remarks \ref{rem:sharpness1}, \ref{SRemcoer}, \ref{TRem} and \ref{rem:sharpness7} below suggest that these upper bounds on  $\mathrm{cond}\,S_k$ and, in the case $n=2$, also the upper bound on $\mathrm{cond}\,T_k$, are sharp in their dependence on $k$.
\section{Preliminaries}
\label{sec:prelims}
In this section we define the Sobolev spaces, trace and jump operators, layer potentials and boundary integral operators studied in the paper.
\subsection{Sobolev spaces}
\label{sec:SobolevSpaces}
Our analysis is in the context of the Sobolev spaces $H^{s} (\Gamma)$ and $\tilde{H}^s(\Gamma)$ for $s\in\R$. We set out here our notation and the basic definitions; for more detail (especially when $\Gamma$ is non-regular) see \cite{ChaHewMoi:13} and \cite[\S2]{CoercScreen}.
Given $n\in \N$, let $\scrD(\R^n):=C^\infty_0(\R^n)$ denote the space of compactly supported smooth test functions on $\R^n$, and let $\scrS(\R^n)$ denote the Schwartz space of rapidly decaying smooth test functions.
For $s\in \R$ let $H^s(\R^n)$ denote the Bessel potential space of those tempered distributions $u\in \scrS^*(\R^n)$ (continuous \emph{antilinear} functionals on $\scrS(\R^n)$) whose Fourier transforms are locally integrable and satisfy
\mbox{$\int_{\R^n}(1+|\bxi|^2)^{s}\,|\hat{u}(\bxi)|^2\,\rd \bxi < \infty.$} (For $s\geq0$ these spaces can be defined equivalently in terms of integrability of weak partial derivatives, the link between the two definitions relying on Plancherel's theorem - see e.g.\ \cite[Theorem 3.16]{McLean} and equation \rf{eqn:H1knorm} below.)
Our convention for the Fourier transform is that
\mbox{$\hat{u}(\bxi) := (2\pi)^{-n/2}\int_{\R^n}\re^{-\ri \bxi\cdot \bx}u(\bx)\,\rd \bx$}, for $u\in \scrS(\R^n)$ and $\bxi\in\R^n$.
In line with many other analyses of high frequency scattering, e.g., \cite{Ih:98}, %
we work with wavenumber-dependent norms. Specifically, we use the norm on $H^s(\R^n)$ defined by
\begin{align}
\norm{u}{H_k^{s}(\R^n)}^2 :=\int_{\R^n}(k^2+|\bxi|^2)^{s}\,|\hat{u}(\bxi)|^2\,\rd \bxi.
\end{align}
We emphasize that $\norm{\cdot}{H^{s}(\R^n)}:= \norm{\cdot}{H_1^{s}(\R^n)}$ is the standard norm on $H^s(\R^n)$, and
that, for $k>0$, $\norm{\cdot}{H_k^{s}(\R^n)}$ is another, equivalent, norm on $H^s(\R^n)$. Explicitly,%
\begin{equation} \label{eq:normequiv}
  \min\{1,k^s\}\norm{u}{H^{s}(\R^n)} \leq \norm{u}{H^{s}_k(\R^n)} \leq \max\{1,k^s\}  \norm{u}{H^{s}(\R^n)},  \mbox{  for }u \in H^s(\R^n).
\end{equation}
The use of $\norm{\cdot}{H_k^{s}(\R^n)}$ instead of $\norm{\cdot}{H_1^{s}(\R^n)}$ in high frequency scattering applications is natural because solutions $u$ of \rf{eq:he} typically oscillate with a wavelength inversely proportional to $k$. Thus, when $k$ is large, $m$th partial derivatives of $u$ are generically $\ord{k^m}$ times larger in absolute value than $u$ itself, which, for $s>0$, leads to $\norm{\cdot}{H_1^{s}(\R^n)}$ being dominated by the behaviour of the higher derivatives of $u$, which is usually undesirable. The $k^2$ appearing in the definition of $\norm{\cdot}{H_k^{s}(\R^n)}$ redresses the balance between lower and higher derivatives. Concretely, in the case $s=1$ it holds by Plancherel's theorem that
\begin{align}
\label{eqn:H1knorm}
\norm{u}{H_k^{1}(\R^n)}^2 =k^2\int_{\R^n}|u(\bx)|^2\,\rd \bx + \int_{\R^n}|\nabla u(\bx)|^2\,\rd \bx, \qquad u\in H^1(\R^n).
\end{align}
If $k$ is large, with $u=\ord{1}$ and $\nabla u=\ord{k^{-1}}$, then the two terms on the right-hand side of \rf{eqn:H1knorm} will be approximately in balance. 
In light of \rf{eqn:H1knorm}, the use of $k$-dependent norms is also natural from a physical point of view: for a solution $u$ of \rf{eq:he} representing an acoustic wave field and a bounded open set $\Omega\subset\R^n$ the quantity
\begin{align*}
k^2\int_{\Omega}|u(\bx)|^2\,\rd \bx + \int_{\Omega}|\nabla u(\bx)|^2\,\rd \bx
\end{align*}
is proportional to the (time-averaged) acoustic energy contained in $\Omega$.

It is standard that $\scrD(\R^n)$ %
is dense in $H^s(\R^n)$.
It is also standard (see, e.g.,~\cite{McLean}) that $H^{-s}(\R^n)$ is a natural isometric realisation of $(H^s(\R^n))^*$, the dual space of bounded antilinear functionals on $H^s(\R^n)$, in the sense that the mapping $u\mapsto u^*$  from $H^{-s}(\R^n)$ to $(H^s(\R^n))^*$, defined by
\begin{align}
\label{DualDef}
u^*(v) := \left\langle u, v \right\rangle_{H^{-s}(\R^n)\times H^{s}(\R^n)}:= \int_{\R^n}\hat{u}(\bxi) \overline{\hat{v}(\bxi)}\,\rd \bxi, \quad v\in H^s(\R^n),
\end{align}
is a unitary isomorphism. The duality pairing $\left\langle \cdot, \cdot \right\rangle_{H^{-s}(\R^n)\times H^{s}(\R^n)}$ defined in \rf{DualDef} represents a natural extension of the $L^2(\R^n)$ inner product in the sense that if $u_j, v_j\in L^2(\R^n)$ for each $j$ and $u_j\to u$ and $v_j\to v$ as $j\to\infty$, with respect to the norms on $H^{-s}(\R^n)$ and $H^{s}(\R^n)$ respectively, then $\left\langle u, v \right\rangle_{H^{-s}(\R^n)\times H^{s}(\R^n)}=\lim_{j\to\infty}\left(u_j, v_j \right)_{L^2(\R^n)}$.

We define two Sobolev spaces on $\Omega$ when $\Omega$ is a non-empty open subset of $\R^n$. First, let $\scrD(\Omega):=C^\infty_0(\Omega)=\{u\in\scrD(\R^n):\supp{U}\subset\Omega\}$, and let $\scrD^*(\Omega)$ denote the associated space of distributions (continuous antilinear functionals on $\scrD(\Omega)$).  We set
$$
H^s(\Omega):=\{u\in\scrD^*(\Omega): u=U|_\Omega \mbox{ for some }U\in H^s(\R^n)\},
$$
where $U|_\Omega$ denotes the restriction of the distribution $U$ to $\Omega$ (cf.\ %
\cite[p.~66]{McLean}), with norm%
\begin{align*}
\|u\|_{H_k^{s}(\Omega)}:=\inf_{U\in H^s(\R^n), \,U|_{\Omega}=u}\normt{U}{H_k^{s}(\R^n)}.
\end{align*}
Then $\scrD(\overline\Omega):=\{u\in C^\infty(\Omega):u=U|_\Omega \textrm{ for some }U\in\scrD(\R^n)\}$ is dense in $H^s(\Omega)$.
Second, let
\begin{align*}
\label{}
\tilde{H}^s(\Omega):=\overline{\scrD(\Omega)}^{H^s(\R^n)}
\end{align*}
denote the closure of $\scrD(\Omega)$ in the space $H^s(\R^n)$, equipped with the norm %
$\|\cdot\|_{\tilde H^s_k(\Omega)}$:= $\|\cdot\|_{H^s_k(\R^n)}$.
When $\Omega$ is sufficiently regular (e.g.\ when $\Omega$ is $C^0$, cf.\ \cite[Thm 3.29]{McLean}) %
we have that $\tilde H^s(\Omega)=H^s_{\overline\Omega}:=\{u\in H^s(\R^n):\supp{u}\subset\overline{\Omega}\}$. (But for non-regular $\Omega$ $\tilde H^s(\Omega)$ may be a proper subset of $H^s_{\overline\Omega}$, see \cite{ChaHewMoi:13}.)

For $s\in \R$ and $\Omega$ any open, non-empty subset of $\R^n$ it holds that
\begin{align}
\label{isdual}
H^{-s}(\Omega)=(\tilde{H}^s(\Omega))^* \; \mbox{ and }\; \tilde H^{s}(\Omega)=(H^{-s}(\Omega))^*,
\end{align}
in the sense that the natural embeddings
$\mathcal{I}:H^{-s}(\Omega)\to(\tilde{H}^s(\Omega))^*$
and
$\mathcal{I}^*:\tilde{H}^{s}(\Omega)\to(H^{-s}(\Omega))^*$,
\begin{align*}
(\mathcal{I} u)(v)
& := \langle u,v \rangle_{H^{-s}(\Omega)\times \tilde{H}^s(\Omega)}:=\langle U,v \rangle_{H^{-s}(\R^n)\times H^s(\R^n)},\\
 (\mathcal{I}^*v)(u)
&:= \langle v,u \rangle_{\tilde H^{s}(\Omega)\times H^{-s}(\Omega)}:=\langle v,U \rangle_{H^{s}(\R^n)\times H^{-s}(\R^n)},
\end{align*}
where $U\in H^{-s}(\R^n)$ is any extension of $u\in H^{-s}(\Omega)$ with $U|_\Omega=u$, are unitary isomorphisms.
We remark that the representations \eqref{isdual} for the dual spaces are well known when $\Omega$ is sufficiently regular. However, it does not appear to be widely appreciated, at least in the numerical analysis for PDEs community, that \eqref{isdual} holds without constraint on the geometry of~$\Omega$, proof of this given recently in \cite[Thm 2.1]{CoercScreen}.

As alluded to above, Sobolev spaces can also be defined, for $s\geq 0$, as subspaces of $L^2(\R^n)$ satisfying constraints on weak derivatives. In particular, given a non-empty open subset $\Omega$ of $\R^n$, let
$$
W^1(\Omega) := \{u\in L^2(\Omega): \nabla u \in L^2(\Omega)\},
$$
where $\nabla u$ is the weak gradient. An obvious consequence of \rf{eqn:H1knorm} is that $W^1(\R^n)=H^1(\R^n)$.  
Further \cite[Theorem 3.30]{McLean}, $W^1(\Omega)= H^1(\Omega)$ whenever $\Omega$ is a Lipschitz open set. %
It is convenient to define
$$
W^1_{\mathrm{loc}}(\Omega) := \{u\in L^2_{\mathrm{loc}}(\Omega): \nabla u \in L^2_{\mathrm{loc}}(\Omega)\},
$$
where $L^2_{\mathrm{loc}}(\Omega)$ denotes the set of locally integrable functions $u$ on $\Omega$ for which $\int_G|u(\bx)|^2 \rd \bx < \infty$ for every bounded measurable $G\subset \Omega$.

To define Sobolev spaces on the screen $\Gamma \subset \Gamma_\infty:=\{\bx=(x_1,...,x_n)\in \R^n:x_n=0\}$ we make the natural associations of $\Gamma_\infty$ with $\R^{n-1}$ and of $\Gamma$ with $\tGamma:=\{\tilde \bx\in\R^{n-1}: (\tilde \bx,0)\in\Gamma\}$ and set
$H^s(\Gamma_\infty) := H^s(\R^{n-1})$, $H^{s}(\Gamma):=H^{s}(\tGamma)$ and $\tilde{H}^{s}(\Gamma):=\tilde{H}^{s}(\tGamma)$
(with $C^\infty(\Gamma_\infty)$, $\scrD(\Gamma_\infty)$, $\scrD(\overline\Gamma)$ and $\scrD(\Gamma)$ defined analogously).

\subsection{Traces, jumps and boundary conditions}
Letting $U^+:=\{\bx\in\R^n:x_n>0\}$ and $U^-:=\R^n\setminus \overline{U^+}$ denote the upper and lower half-spaces, respectively, %
we define trace operators  $\gamma^\pm:\scrD(\overline{U^\pm})\to \scrD(\Gamma_\infty)$ by $\gamma^\pm u := u|_{\Gamma_\infty}$. It is well known that these extend to bounded linear operators $\gamma^\pm:W^1(U^\pm)\to H^{1/2}(\Gamma_\infty)$.
Similarly, we define normal derivative operators $\partial_{\bn}^\pm:\scrD(\overline{U^\pm})\to \scrD(\Gamma_\infty)$ by
$\partial_{\bn}^\pm u = \pdonetext{u}{x_n}|_{\Gamma_\infty}$ (so the normal points into $U^+$), which
extend  (see, e.g.,\ \cite{ChGrLaSp:11}) to bounded linear operators $\partial_\bn^\pm:W^1(U^\pm;\Delta)\to H^{-1/2}(\Gamma_\infty)=(H^{1/2}(\Gamma_\infty))^*$, where $W^1(U^\pm;\Delta):= \{u\in H^1(U^\pm):\Delta u\in L^2(U^\pm)\}$ and $\Delta u$ is the weak Laplacian.

To define jumps across $\Gamma$, let $u\in L^2_{\rm loc}(\R^n)$ be such that $u|_{U^\pm}\in W_{\rm loc}^1(U^\pm;\Delta)$ with
\begin{align*}
\label{}
\gamma^+(\chi u)|_{\Gamma_\infty\setminus\overline\Gamma} - \gamma^+(\chi u)|_{\Gamma_\infty\setminus\overline\Gamma} = 0 \quad \mbox{ and } \quad
\partial_\bn^+(\chi u)|_{\Gamma_\infty\setminus\overline\Gamma} - \partial_\bn^+(\chi u)|_{\Gamma_\infty\setminus\overline\Gamma} = 0
\end{align*}
for all $\chi\in\scrD(\R^n)$. We then define
\begin{align*}
\label{}
[u]&:=\gamma^+(\chi u)-\gamma^-(\chi u)\in \tH^{1/2}(\Gamma) \\ 
[\pdonetext{u}{\bn}]&:=\partial^+_\bn(\chi u)-\partial^-_\bn(\chi u)\in \tH^{-1/2}(\Gamma),
\end{align*}
where $\chi$ is any element of $\scrD_{1,\Gamma}(\R^n):=\{\phi \in \scrD(\R^n)$: $\phi=1$ in some neighbourhood of $\Gamma\}$.

The boundary conditions \rf{eqn:bc1} and \rf{eqn:bc2} can now be stated more precisely: by \rf{eqn:bc1} and \rf{eqn:bc2} we mean that
\[
\gamma^\pm (\chi u)\vert_\Gamma =g_{\sD},
\qquad \mbox{ and } \qquad
\partial_\bn^\pm (\chi u)\vert_\Gamma =g_{\sN},
\qquad 
\mbox{for every } \chi\in\scrD_{1,\Gamma}(\R^n).
\]

\subsection{Layer potentials and boundary integral operators}
\label{sec:PotsBIOs}
We can now give precise definitions for the single and double layer potentials
\begin{align*}
\label{}
\cS_k:\tH^{-1/2}(\Gamma)\to C^2(D)\cap W^1_{\rm loc}(D),\qquad \cD_k:\tH^{1/2}(\Gamma)\to C^2(D)\cap W^1_{\rm loc}(D),
\end{align*}
namely
\begin{align*}
\label{}
\cS_k\phi (\bx)&:=\left\langle \gamma^\pm(\rho\Phi(\bx,\cdot))|_\Gamma,\overline{\phi}\right\rangle_{H^{1/2}(\Gamma)\times \tilde{H}^{-1/2}(\Gamma)}\\
&=\left\langle \gamma^\pm(\rho\Phi(\bx,\cdot)),\overline{\phi}\right\rangle_{H^{1/2}(\Gamma_\infty)\times H^{-1/2}(\Gamma_\infty)}, \qquad \bx\in D,\, \phi\in \tH^{-1/2}(\Gamma),\\
\cD_k\psi (\bx)&:=\left\langle \psi, \overline{\partial^\pm_\bn(\rho\Phi(\bx,\cdot))|_\Gamma}\right\rangle_{\tilde{H}^{1/2}(\Gamma)\times H^{-1/2}(\Gamma)}\\
&=\left\langle \psi, \overline{\partial^\pm_\bn(\rho\Phi(\bx,\cdot))}\right\rangle_{H^{1/2}(\Gamma_\infty)\times H^{-1/2}(\Gamma_\infty)}, \qquad \bx\in D,\, \psi\in \tH^{1/2}(\Gamma),
\end{align*}
where $\rho$ is any element of $\scrD_{1,\Gamma}(\R^n)$ with $\bx\not\in \supp{\rho}$.
The single-layer and hypersingular boundary integral operators
\begin{align*}
\label{}
S_k:\tH^{-1/2}(\Gamma)\to H^{1/2}(\Gamma),\qquad T_k:\tH^{1/2}(\Gamma)\to H^{-1/2}(\Gamma),
\end{align*}
are then defined by
\begin{align*}
S_k\phi :=\gamma^\pm(\chi\cS_k\phi)|_\Gamma, \quad \phi\in \tH^{-1/2}(\Gamma),
\quad
T_k\phi :=\partial^\pm_\bn(\chi\cD_k\psi)|_\Gamma, \quad \psi\in \tH^{1/2}(\Gamma),
\end{align*}
where $\chi$ is any element of $\scrD_{1,\Gamma}(\R^n)$, and either of the $\pm$ traces may be taken. When $\phi,\psi \in \scrD(\Gamma)$, it follows from \cite[p.~202]{McLean} and \cite[Theorems 2.12 and 2.23]{CoKr:83} that $\cS_k\phi$, $\cD_k\psi$, $S_k\phi$, and $T_k\psi$ are given explicitly by \eqref{SkPotIntRep} and \eqref{SDef}.

\section{Fourier representations for layer potentials and BIOs}
\label{sec:FourierRep}
Our approach to proving the $k$-explicit continuity and coercivity results in Theorems \ref{ThmSNormSmooth}-\ref{ThmTCoercive} is to make use of the fact that, because the screen is planar, the single and double layer potentials and the single-layer and hypersingular BIOs can be expressed in terms of Fourier transforms, this observation captured in the following Theorem \ref{ThmFourierRep}. We note that the parts of this theorem relating to $T_k$ were stated and proved for the case $n=3$ in \cite[Theorems 1 and 2]{Ha-Du:92}. For completeness we include below a short direct proof of the whole theorem, which introduces notation and formulae which prove useful in later sections. We remark that an alternative method of proving \rf{SPotFourierRep} and \rf{DPotFourierRep} would be to observe, using elementary arguments and standard properties of single- and double-layer potentials \cite{CoKr:83}, that the left and right hand sides of each equation satisfy the same boundary value problems for the Helmholtz equation in $D$, with the same Neumann data on $\Gamma$ in the case of \eqref{SPotFourierRep}, the same Dirichlet data on $\Gamma$ in the case of \eqref{DPotFourierRep}, so that the right and left hand sides must coincide. (This argument is most easily done for $k$ replaced by $k+\ri\eps$, with $\eps>0$, and then the result for real wavenumber obtained by taking the limit $\eps\to 0^+$, using the dominated convergence theorem.) 
\begin{thm}
\label{ThmFourierRep}
Let $\phi\in \scrD (\Gamma)$. Then
\begin{align}
\label{SPotFourierRep}
\cS_k\phi(\bx)&=\frac{\ri}{2(2\pi)^{(n-1)/2}}\int_{\R^{n-1}}\frac{\re^{\ri (\bxi\cdot \tbx+|x_n| Z(\bxi))}}{Z(\bxi)}\widehat{\varphi}(\bxi)\,\rd \bxi,  & &\bx=(\tbx,x_n)\in D,\\
\label{DPotFourierRep}
\cD_k\phi(\bx)&=\frac{\sign{x_n}}{2(2\pi)^{(n-1)/2}}\int_{\R^{n-1}}\re^{\ri (\bxi\cdot \tbx+|x_n| Z(\bxi))}\widehat{\varphi}(\bxi)\,\rd \bxi, & & \bx=(\tbx,x_n)\in D,
\end{align}
where $\hat{}$ represents the Fourier transform with respect to $\tbx\in \R^{n-1}$ and
\begin{align}
\label{ZDef}
Z(\bxi):=
\begin{cases}
\sqrt{k^2-|\bxi|^2}, & |\bxi|\leq k\\
\ri\sqrt{|\bxi|^2-k^2}, & |\bxi|> k,
\end{cases}
\qquad \bxi\in \R^{n-1}.
\end{align}
The operators $S_k,T_k:\scrD (\Gamma) \to \scrD (\overline\Gamma)$ satisfy $S_k\phi = (S_k^\infty \phi)|_\Gamma$ and $T_k\phi = (T_k^\infty \phi)|_\Gamma$, where $S_k^\infty, T_k^\infty:\scrD(\R^{n-1})\to C^\infty(\R^{n-1})$ are the pseudodifferential operators defined for $\varphi\in \scrD(\R^{n-1})$ by
\begin{align}
\label{SFourierRep}
S_k^\infty\varphi(\tbx)&=\frac{\ri}{2(2\pi)^{(n-1)/2}}\int_{\R^{n-1}}\frac{\re^{\ri \bxi\cdot \tbx}}{Z(\bxi)}\widehat{\varphi}(\bxi)\,\rd \bxi, & & \tbx\in\R^{n-1},\\
\label{TFourierRep}
T_k^\infty\varphi(\tbx)&=\frac{\ri}{2(2\pi)^{(n-1)/2}}\int_{\R^{n-1}}Z(\bxi)\re^{\ri \bxi\cdot \tbx}\widehat{\varphi}(\bxi)\,\rd \bxi, & & \tbx\in\R^{n-1}.%
\end{align}
Furthermore, for $\phi, \psi\in \scrD(\Gamma)$ we have that
\begin{align}
\label{SIPFourierRep}
(S_k\phi,\psi)_{L^2(\Gamma)} &= \frac{\ri}{2 }\int_{\R^{n-1}} \frac{1}{Z(\bxi)} \widehat{\phi}(\bxi) \overline{\widehat{\psi}(\bxi)} \,\rd \bxi,\\
\label{TIPFourierRep}
(T_k\phi,\psi)_{L^2(\Gamma)} &= \frac{\ri}{2 }\int_{\R^{n-1}} Z(\bxi) \widehat{\phi}(\bxi) \overline{\widehat{\psi}(\bxi)} \,\rd \bxi.
\end{align}
\end{thm}

\begin{proof} %
Using \rf{SkPotIntRep}, we see that, for $\phi\in\scrD(\Gamma)$, %
\begin{align}
\label{}
\cS_k\phi(\bx) = (\Phi_c(\cdot,x_n)\ast \phi)(\tbx),
\end{align}
where $\ast$ indicates a convolution over $\R^{n-1}$ (with $x_n$ treated as a parameter) and
\begin{align*}
\label{}
\Phi_c(\tbx,x_n) :=\Phi((\tbx ,x_n),\bs{0}) =
\begin{cases}
\dfrac{\re^{\ri k\sqrt{r^2+x_n^2}}}{4\pi \sqrt{r^2+x_n^2}}, & n=3,\\[3mm]
\dfrac{\ri }{4}H_0^{(1)}(k\sqrt{r^2+x_n^2}), & n=2,
\end{cases}
\qquad  r=|\tbx |,\, \tbx \in\R^{n-1}.
\end{align*}
Hence the Fourier transform (with respect to $\tbx\in\R^{n-1}$) of $\cS_k \phi$ is given by the product
\begin{align*}
\widehat{\cS_k \phi}(\bxi,x_n) = (2\pi)^{(n-1)/2}\,\widehat{\Phi_c}(\bxi,x_n)\hat{\varphi}(\bxi).
\end{align*}
To evaluate $\widehat{\Phi_c}$ we note that for a function $f(\bx )=F(r)$, where $r=|\bx |$ for $\bx \in\R^{d}$, $d=1,2$, the Fourier transform of $f$ is given by (cf.\ \cite[\S{B.5}]{Grafakos})
\footnote{Strictly speaking, \cite[\S{B.5}]{Grafakos} only provides \rf{FTRadial} for $f\in L^1(\R^d)$. %
But for the functions $f=\Phi_c(\cdot,x_n)$ one can check using the dominated convergence theorem that \rf{FTRadial} holds. %
}
\begin{align}
\label{FTRadial}
\hat{f}(\bxi)=
\begin{cases}
\displaystyle{\int_0^\infty F(r)J_0(|\bxi|r)r\,\rd r}, & d=2,\\[3mm]
\displaystyle{
\sqrt{\frac{2}{\pi}}\int_0^\infty F(r)\cos(\bxi r)\,\rd r}, & d=1.
\end{cases}
\end{align}
This result, combined with the identities \cite[(6.677), (6.737)]{Ryzhik}
and \cite[(10.16.1), (10.39.2)]{DLMF},
gives
\begin{align*}
\label{}
\widehat{\Phi_c}(\bxi,x_n)= \dfrac{\ri \,\re^{\ri |x_n| Z(\bxi)}}{2(2\pi)^{(n-1)/2}Z(\bxi)},
\end{align*}
where $Z(\bxi)$ is defined as in \rf{ZDef}. The representation \rf{SFourierRep} is then obtained by Fourier inversion.

The representation \rf{DPotFourierRep} for $\cD_k\phi$ can be then obtained from \rf{SPotFourierRep} by noting that
\begin{align*}
\label{}
\pdone{\Phi(\bx,\by)}{\bn(y)}=\pdone{\Phi(\bx,\by)}{y_n} = -\pdone{\Phi(\bx,\by)}{x_n}, \qquad \bx\in D, \,\by\in \Gamma,
\end{align*}
and the representations for $S_k$ and $T_k$ follow from taking the appropriate traces of \rf{SPotFourierRep} and \rf{DPotFourierRep}.%

Finally, \rf{SIPFourierRep} and \rf{TIPFourierRep} follow from viewing $S^\infty_k\phi$ and $T^\infty_k\phi$ as elements of $C^\infty(\R^{n-1})\cap \mathscr{S}^*(\R^{n-1})$ and recalling the definition of the Fourier transform of a distribution, e.g., for $S_k$,%
\begin{align*}
(S_k\phi,\psi)_{L^2(\Gamma)}
=  \int_{\R^{n-1}} S_k^\infty\phi(\tbx)\overline{\psi(\tbx)}\, d\tbx
&= \int_{\R^{n-1}} \widehat{S_k^\infty\phi}(\bxi)\overline{\widehat{\psi}(\bxi)}\, d\bxi \\
&= \frac{\ri}{2 }\int_{\R^{n-1}} \frac{1}{Z(\bxi)} \widehat{\phi}(\bxi) \overline{\widehat{\psi}(\bxi)} \,\rd \bxi.
\end{align*}
\end{proof}
\section{$k$-explicit analysis of $S_k$}
\label{sec:SkAnalysis}
Our $k$-explicit analysis of the single-layer operator $S_k$ makes use of the following lemma.
\begin{lem} %
\label{LemPhiFT}
Given $L>0$ let
\begin{align}
\label{PhiADef}
\Phi_L(\tbx ,x_n) : =
\begin{cases}
\Phi_c(\tbx,x_n), & |\tbx |\leq L,\\
0, & |\tbx |>L,
\end{cases}
\end{align}
where $\Phi_c$ is defined as in the proof of Theorem \ref{ThmFourierRep}.
Then there exists a constant $C>0$, independent of $k$, $L$, $\bxi$ and $x_n$, such that, for all $k>0$, $\bxi\in\R^{n-1}$, and $x_n\in \R$,
\begin{align}
\label{EquivBounded}
|\widehat{ \Phi_L}(\bxi,x_n)|\sqrt{k^2+|\bxi|^2} \leq
\begin{cases}
C(1+ (kL)^{1/2}), & n=3,\\
C \left( \log(2+(kL)^{-1}) + (kL)^{1/2}\right), & n=2.\\
\end{cases}
\end{align}

\end{lem}
\begin{proof}
It is convenient to introduce the notation $\txi :=|\bxi|$, and by $C>0$ we denote an arbitrary constant, independent of $k$, $L$, $\bxi$, and $x_n$, which may change from occurrence to occurrence. To prove \rf{EquivBounded} we proceed by estimating $|\widehat{ \Phi_L}(\bxi,x_n)|$ directly, using the formula \rf{FTRadial}. We treat the cases $n=3$ and $n=2$ separately. We will make use of the following well-known properties of the Bessel functions (cf.\ \cite[Sections 10.6, 10.14, 10.17]{DLMF}), where $\cB_n$ represents either $J_n$ or $H_n^{(1)}$:
\begin{align}
\label{JnEst}
|J_n(z)|\leq 1,& & & n\in\N,\,\,z>0,\\
\label{H0Est}
|H_0^{(1)}(z)|\leq C(1+|\log{z}|),& & & 0<z\leq 1\\
\label{H1Est}
|H_1^{(1)}(z)|\leq Cz^{-1},& & & 0<z\leq 1\\
\label{H1Est2}
\left|H_1^{(1)}(z)+\frac{2\ri}{\pi z}\right| \leq Cz^{-1/2}, & & & z>0,\\
\label{BnEst}
|\cB_n(z)|\leq Cz^{-1/2},& & & n=0,1,\,\,z>1,\\
\label{B0Der}
\cB_0'(z)= - \cB_1(z), & & & z>0,\\
\label{B1Der}
\frac{d}{dz}(z\cB_1(z))= z\cB_0(z), & & & z>0.
\end{align}

(i) In the case $n=3$, $|\widehat{ \Phi_L}(\bxi,x_3)|\leq |I(L)|/(4\pi)$,
where
\begin{align*}
\label{}
I(L):=\int_0^{L} \frac{\re^{\ri \tk \sqrt{r^2+x_3^2}}}{\sqrt{r^2+x_3^2}}J_0(\txi r)\,r\,\rd r, \quad \mbox{for }L>0.
\end{align*}
Using \eqref{JnEst}, we see that
$|I(L)| \leq 1/\txi$, if $L\leq 1/\txi$.
If $L>1/\txi$ then,
integrating by parts using the relation \rf{B1Der}, %
\begin{align*}
\label{}
I(L)-I(1/\txi) &= \frac{1}{\txi}\left[\frac{r \re^{\ri \tk \sqrt{r^2+x_3^2}}}{\sqrt{r^2+x_3^2}}J_1(\txi r)\right]^L_{1/\txi} \\
&\qquad -\frac{1}{\txi} \int_{1/\txi}^L r^2 \re^{\ri \tk \sqrt{r^2+x_3^2}}\left(\frac{\ri k}{r^2+x_3^2}-\frac{1}{(r^2+x_3^2)^{3/2}}\right) J_1(\txi r) \,\rd r,
\end{align*}
so that, substituting $t=\xi r$ and using \eqref{JnEst},
\begin{align}
\label{eq:Ibound2}
|I(L)| \leq |I(1/\xi)| + |I(L)-I(1/\txi)| \leq \frac{3}{\txi} +\frac{1}{\txi} \int_{1}^{\txi L} \left(\frac{k}{\txi}+t^{-1}\right) |J_1(t)| \,\rd t.
\end{align}
Using the bound \eqref{BnEst} in \eqref{eq:Ibound2}, it follows that
$$
|I(L)| \leq \frac{3}{\txi} +\frac{C}{\txi}  \left(\frac{kL^{1/2}}{\txi^{1/2}}+1\right), \qquad L>0,
$$
so that
\begin{align}
\label{PEst3}
|\widehat{ \Phi_L}(\bxi,x_3)|\sqrt{k^2+\xi^2}\leq C(1+ (kL)^{1/2}), \qquad \mbox{for } 0<\tk <\txi.
\end{align}

On the other hand, integrating by parts using the relation \rf{B0Der},%
\begin{align*}
\label{}
I(L) = \frac{1}{\tk}\left[-\ri \re^{\ri \tk \sqrt{r^2+x_3^2}}J_0(\txi r)\right]_0^{L} - \frac{\ri \txi}{\tk}\int_{0}^L \re^{\ri \tk \sqrt{r^2+x_3^2}}J_1(\txi r)  \,\rd r,
\end{align*}
so that, substituting $t=\xi r$ and using \eqref{JnEst},
\begin{align*}
\label{}
|I(L)| \leq \frac{2}{\tk} + \frac{1}{\tk}\int_{0}^{kL} |J_1(t)|  \,\rd t.
\end{align*}
Using \eqref{JnEst} and \eqref{BnEst} we see that \eqref{PEst3} holds also for  $0\leq\txi\leq\tk$, establishing \rf{EquivBounded} in the case $n=3$.

(ii) In the case $n=2$, $|\widehat{ \Phi_L}(\bxi,x_2)| \leq |I(L)|/(2\sqrt{2\pi})$, %
where now
\begin{align*}
\label{}
I(L) :=\int_0^{L} H_0^{(1)}(\tk \sqrt{r^2+x_2^2})\cos(\txi r)\,\rd r.
\end{align*}
Using the monotonicity of $|H_0^{(1)}(z)|$ for $z>0$ \cite[p.~487]{Wa:44}, we see that
\begin{equation} \label{eq:Ibound2D}
|I(L)| \leq \int_0^{L} |H_0^{(1)}(\tk r)|\,\rd r = \frac{1}{k} \int_0^{kL} |H_0^{(1)}(t)|\,\rd t, \qquad L>0.
\end{equation}
Using \eqref{H0Est} and \eqref{BnEst}, we deduce that
\begin{align}
\label{PEst1a}
|\widehat{ \Phi_L}(\bxi,x_2)|\sqrt{k^2+\xi^2}\leq C\left(\log(2+(kL)^{-1}) + (kL)^{1/2}\right), \qquad \mbox{for } 0\leq \txi\leq k.
\end{align}

Further, integrating by parts using the relation \rf{B0Der} gives
\begin{align*}
I(L) &= \frac{1}{\txi}\left[H_0^{(1)}(\tk \sqrt{r^2+x_2^2}) \sin{\txi r}\right]_0^{L}  \\
&\qquad\qquad+ \frac{\tk}{\txi}\int_{0}^L \frac{r}{\sqrt{r^2+x_2^2}}H_1^{(1)}(\tk \sqrt{r^2+x_2^2}) \sin \txi r \,\rd r\\
& =  \frac{1}{\txi}H_0^{(1)}(\tk \sqrt{L^2+x_2^2}) \sin{\txi L}  \\
&\qquad\qquad+ \frac{\tk}{\txi}\int_{0}^L \frac{r}{\sqrt{r^2+x_2^2}}F_1(\tk \sqrt{r^2+x_2^2}) \sin \txi r\, \rd r + I_0(L),
\end{align*}
where $c_0 := -2\ri/\pi$, $F_1(z):= H_1^{(1)}(z) - c_0/z$, and
\begin{align*}
\label{}
I_0(L) := \frac{c_0}{\txi}\int_{0}^L \frac{r \sin \txi r}{r^2+x_2^2} \,\rd r = \frac{c_0}{\txi}\int_{0}^{\txi L} \frac{t \sin t}{t^2+\txi^2x_2^2} \,\rd t.
\end{align*}
Now $|I_0(L)| \leq |c_0|/\txi= 2/(\pi\txi)$, for $L \leq 1/\txi$. For $L > 1/\txi$, integrating by parts,
$$
I_0(L) -I_0(1/\txi) %
= \frac{c_0}{\txi}\left(\left[\frac{-t\cos t}{t^2+\txi^2 x_2^2}\right]_1^{\txi L} + \int_{1}^{\txi L} \frac{(\txi^2x_2^2-t^2) \cos t}{(t^2+\txi^2x_2^2)^2} \,\rd t\right),
$$
so that
\begin{equation*} %
|I_0(L)| \leq \frac{2}{\pi\txi}\left(3 + \int_{1}^{\txi L} \frac{1}{t^2} \,\rd t\right) <\frac{8}{\pi\txi}.
\end{equation*}
Using these bounds on $I_0(L)$ and the bound \eqref{H1Est2} on $F_1(z)$, we see that
\begin{align*}
|I(L)| &\leq \frac{1}{\txi}|H_0^{(1)}(k L)| + \frac{C\tk^{1/2}}{\txi}\int_{0}^L r^{-1/2} \rd r + \frac{8}{\pi\txi} \\ \qquad &= \frac{1}{\txi}\left(|H_0^{(1)}(k L)| + 2C(\tk L)^{1/2} + \frac{8}{\pi}\right),
\end{align*}
for $L>0$. Using the bounds \eqref{H0Est} and \eqref{BnEst}, we conclude that \rf{PEst1a} holds also for $0<\tk<\txi$, establishing \rf{EquivBounded} in the case $n=2$.
\end{proof}

Using this result we can now prove Theorem \ref{ThmSNormSmooth}.

\begin{proof}[Proof of Theorem \ref{ThmSNormSmooth}]
By the density of $\scrD(\Gamma)$ in $\tilde{H}^s(\Gamma)$ it suffices to prove \rf{SEst} for $\phi\in \scrD(\Gamma)$. For $\phi\in \scrD(\Gamma)$ we first note that $S_k\phi = (S_k^L \phi)|_\Gamma$, where $S_k^L:\scrD(\R^{n-1})\to \scrD(\R^{n-1})$ is the convolution operator
defined by
$S_k^L\varphi: = (\Phi_L(\cdot,0)\ast \varphi)$, for $\varphi\in \scrD(\R^{n-1})$, 
where $\Phi_L$ is defined as in \rf{PhiADef}.
While $S_k^\infty\varphi \in C^\infty(\R^{n-1})$ for $\varphi\in \scrD(\R^{n-1})$, the fact that $\Phi_L$ has compact support means that $S_k^L\varphi\in \scrD(\R^{n-1})\subset H^{s+1}(\R^{n-1})$ (cf.\ \cite[Corollary 5.4-2a]{Zemanian}).
Therefore, for $\phi\in \scrD(\Gamma)$ we can estimate
$\norm{S_k \phi}{H^{s+1}_k(\Gamma)}\leq \normt{S_k^L \phi}{H^{s+1}_k(\R^{n-1})}$,
and since
$\widehat{S_k^L \varphi}(\bxi) = \widehat{( \Phi_L(\cdot,0)\ast \varphi)}(\bxi) = (2\pi)^{(n-1)/2}\widehat{ \Phi_L}(\bxi,0)\hat{\varphi}(\bxi)$
for any $\varphi\in \scrD(\R^{n-1})$,
the bound \rf{SEst} follows from Lemma \ref{LemPhiFT}.
\end{proof}

\begin{proof}[Proof of Theorem \ref{ThmCoerc}] 
By the density of $\scrD(\Gamma)$ in $\tilde{H}^{-1/2}(\Gamma)$ it suffices to prove \rf{SCoercive} for $\phi\in \scrD(\Gamma)$.  For such a $\phi$,  noting that $\real{\exp(\ri\pi/4)/Z(\bxi)} = |\sqrt{k^2-|\bxi|^2}\,|/\sqrt{2}$, formula \rf{SIPFourierRep} from Theorem \ref{ThmFourierRep} gives the desired result:
\begin{align}
 |a_\sD(\phi,\phi)|   = |(S_k\phi,\phi)_{L^2(\Gamma)}| %
& \geq %
\frac{1}{2\sqrt{2}} \int_{\R^{n-1}} \frac{|\widehat{\phi}(\bxi)|^2}{\sqrt{|k^2-|\bxi|^2|}} \,\rd \bxi \notag \\
&  \geq\frac{1}{2\sqrt{2}} \int_{\R^{n-1}} \frac{|\widehat{\phi}(\bxi)|^2}{\sqrt{k^2+|\bxi|^2}} \,\rd \bxi.
\label{aCoercProof}
\end{align}
\end{proof}

\begin{rem}\label{rem:sharpness1}%
We can show that the bounds established in Theorem \ref{ThmSNormSmooth} are sharp in their dependence on $k$ as $k\to\infty$. For simplicity of presentation we assume that $\diam \Gamma=1$ and $k>1$. Let $\phi(\tbx):=\re^{\ri k\tbd\cdot \tbx}\psi(\tbx)$ for $\tbx\in\R^{n-1}$, where $\tbd\in \R^{n-1}$ is a unit vector and $0\neq \psi\in \scrD(\Gamma)$ is independent of $k$.  %
Then $\widehat{\phi}(\bxi) = \widehat{\psi}(\boldsymbol{\eta})$, where $\boldsymbol{\eta}=\bxi-k\tbd$. Thus, for any $\eta_*\geq 1$, where $I(\eta_*) := \int_{|\boldsymbol{\eta}|\leq \eta_*} |\widehat{\psi}(\boldsymbol{\eta})|^2\,\rd \boldsymbol{\eta}$, substituting $\bxi = \boldsymbol{\eta}+k\tbd$,
the first inequality in \rf{aCoercProof} gives that
\begin{align}
\label{CRem1}
|(S_k\phi,\phi)_{L^2(\Gamma)}| \geq %
\frac{1}{2\sqrt{2} } \int_{\R^{n-1}} \frac{|\widehat{\psi}(\boldsymbol{\eta})|^2}{\sqrt{|k^2-|\bxi|^2|}} \,\rd \boldsymbol{\eta}
\geq \frac{I(\eta_*)}{2\sqrt{6}\eta_*k^{1/2}},
\end{align}
since $|k^2-|\bxi|^2| \leq 2k |\boldsymbol{\eta}| + |\boldsymbol{\eta}|^2  \leq3k\eta_*^2$,
for $|\boldsymbol{\eta}|\leq \eta_*$. %
Also, for the same choice of $\phi$ and all $t\geq0$,
\begin{align}
\label{CRem2}
 \norm{\phi}{\tilde{H}^{-t}_k(\Gamma)}^2 \leq \frac{1}{k^{2t}} \int_{\R^{n-1}} |\widehat{\phi}(\bxi)|^2 \,\rd \bxi
= \frac{1}{k^{2t}} \int_{\R^{n-1}} |\widehat{\psi}(\boldsymbol{\eta})|^2\,\rd \boldsymbol{\eta}
\leq  2k^{-2t} I(\eta_*),
\end{align}
for $\eta_*$ sufficiently large. Further, for $\eta_*$ sufficiently large and $k>1$,
$$
I(\eta_*) \geq  \int_{|\boldsymbol{\eta}|> \eta_*} (2+2|\boldsymbol{\eta}|+|\boldsymbol{\eta}|^2)^t|\widehat{\psi}(\boldsymbol{\eta})|^2\,\rd \boldsymbol{\eta}
\geq  \int_{|\boldsymbol{\eta}|> \eta_*} \left(1+\left|\frac{\boldsymbol{\eta}}{k}+\tbd\right|^2\right)^t|\widehat{\psi}(\boldsymbol{\eta})|^2\,\rd \boldsymbol{\eta},
$$
so that, since $k^2+|\boldsymbol{\eta}+k\tbd|^2 \leq k^2+(\eta_*+k)^2 \leq 5k^2\eta_*^2$ for $|\boldsymbol{\eta}| \leq \eta_*$,
\begin{align} \label{CRem3}
\norm{\phi}{\tilde{H}^{t}_k(\Gamma)}^2 = \int_{\R^{n-1}} \left(k^2+\left|\boldsymbol{\eta}+k\tbd\right|^2\right)^t|\widehat{\psi}(\boldsymbol{\eta})|^2\,\rd \boldsymbol{\eta} \leq 6k^{2t}\eta_*^{2t} I(\eta_*).
\end{align}
Combining \rf{CRem1}, \rf{CRem2} and \rf{CRem3} we see that, for every $s\in \R$, if $\eta_*\geq 1$ is sufficiently large, there exists $C>0$, depending on $\eta_*$ and $s$ but independent of $k$, such that
$$
|(S_k\phi,\phi)_{L^2(\Gamma)}| \geq C k^{1/2} \norm{\phi}{\tilde{H}^{-(s+1)}_k(\Gamma)} \, \norm{\phi}{\tilde{H}^{s}_k(\Gamma)}.
$$
But, on the other hand,
$$
|(S_k\phi,\phi)_{L^2(\Gamma)}| = |\langle S_k\phi,\phi \rangle_{H^{s+1}(\Gamma)\times \tilde{H}^{-(s+1)}(\Gamma)}|\leq  \norm{S_k\phi}{{H}^{s+1}_k(\Gamma)} \, \norm{\phi}{\tilde{H}^{-(s+1)}_k(\Gamma)},
$$
so that, for this particular choice of $\phi$,%
\begin{align*}
\label{}
 \norm{S_k\phi}{H^{s+1}_k(\Gamma)}  \geq C k^{1/2} \norm{\phi}{\tilde{H}^{s}_k(\Gamma)},
\end{align*}
which demonstrates the sharpness of \rf{SEst} in the limit $k\to\infty$.
\end{rem}

\begin{rem} Theorem \ref{ThmSNormSmooth} bounds $S_k:\tilde H^{s}(\Gamma)\to H^{s+1}(\Gamma)$. We can also bound $S_k$ as a mapping $S_k:\tilde H^{s}(\Gamma)\to H^s(\Gamma)$. Since $\norm{\phi}{\tilde H^{s-1}_k(\Gamma)} \leq k^{-1} \norm{\phi}{\tilde H^{s}_k(\Gamma)}$ for $\phi\in \tilde H^s(\Gamma)$, it follows from Theorem \ref{ThmSNormSmooth} that, for $kL\geq 1$,
\begin{align} \label{eq:sequal}
\norm{S_k\phi}{H^{s}_k(\Gamma)} \leq C\left(\frac{L}{k}\right)^{1/2} \norm{\phi}{\tilde H^{s}_k(\Gamma)}, \qquad \mbox{for } \phi\in \tilde H^s(\Gamma).
\end{align}
Arguing as in Remark \ref{rem:sharpness1} above, and with the same choice of $\phi$ and again with $L=\diam(\Gamma)=1$ and assuming $k>1$, we easily see that for every $s\in \R$ there exists $C>0$ such that
$$
\norm{S_k\phi}{{H}^{s}_k(\Gamma)} \, \norm{\phi}{\tilde{H}^{-s}_k(\Gamma)} \geq|(S_k\phi,\phi)_{L^2(\Gamma)}| \geq C k^{-1/2} \norm{\phi}{\tilde{H}^{-s}_k(\Gamma)} \, \norm{\phi}{\tilde{H}^{s}_k(\Gamma)},
$$
so that
\begin{align}
\label{eq:sequal2}
 \norm{S_k\phi}{H^{s}_k(\Gamma)}  \geq C k^{-1/2} \norm{\phi}{\tilde{H}^{s}_k(\Gamma)},
\end{align}
which demonstrates the sharpness of \rf{eq:sequal} in the limit $k\to\infty$.

We note that in the case $s=0$, when $H^s(\Gamma)=\tilde H^{-s}(\Gamma)=L^2(\Gamma)$, \rf{eq:sequal} and  \rf{eq:sequal2} provide upper and lower bounds on the norm of $S_k$ as an operator on $L^2(\Gamma)$, these bounds shown previously in the 2D case in \cite{CWGLL09} and, in the multidimensional case, very recently in \cite{HanTacyGalkowski14}.
\end{rem}

\begin{rem}
\label{SRemcoer}
We can also show that the bound in Theorem \ref{ThmCoerc} is sharp in its dependence on $k$ as $k\to\infty$.  Let $0\neq \phi\in \scrD(\Gamma)$ be independent of $k$. Then, by \rf{SIPFourierRep}, and since $\widehat{\phi}(\bxi)$ is rapidly decreasing as $k\to\infty$,
\begin{align}
|a_\sD(\phi,\phi)| \leq \frac{1}{2} \int_{\R^{n-1}}\frac{|\widehat{\phi}(\bxi)|^2}{\sqrt{|k^2-|\bxi|^2|}} \,\rd \bxi
& \sim  \frac{1}{2k} \int_{\R^{n-1}}\,|\widehat{\phi}(\bxi)|^2 \,\rd \bxi,
\end{align}
as $k\to\infty$. Further,
\begin{align}
 \norm{\phi}{\tilde{H}^{-1/2}_k(\Gamma)}^2
\sim  \frac{1}{k} \int_{\R^{n-1}} |\widehat{\phi}(\bxi)|^2\,\rd \bxi
\end{align}
as $k\to\infty$. Thus, for every $C>1/2$,
\begin{align}
|a_\sD(\phi,\phi)| \leq C
 \norm{\phi}{\tilde{H}^{-1/2}_k(\Gamma)}^2,
\end{align}
for all sufficiently large $k$.
\end{rem}

\section{$k$-explicit analysis of $T_k$}
\label{sec:TkAnalysis}

\begin{proof}[Proof of Theorem \ref{ThmTNormSmooth}]
By the density of $\scrD(\Gamma)$ in $\tilde{H}^s(\Gamma)$ it suffices to prove \rf{TEst} for $\phi\in \scrD(\Gamma)$. For such a $\phi$ we first note from Theorem \ref{ThmFourierRep} that $T_k\phi = (T_k^\infty \phi)|_\Gamma$, where
$\widehat{T_k^\infty\varphi}(\bxi) = (\ri/2) Z(\bxi) \hat{\varphi}(\bxi)$, for $\varphi\in \scrD(\R^{n-1})$.
Clearly, for any $\varphi\in \scrD(\R^{n-1})$ and any $s\in \R$, the integral
\begin{align*}
\label{}
\int_{\R^{n-1}}(k^2+|\bxi|^2)^{s-1} |\widehat{T_k^\infty \varphi}(\bxi)|^2 \, \rd \bxi
\end{align*}
is finite, and hence $T_k^\infty\varphi\in H^{s-1}(\R^{n-1})$. As a result, given $\phi\in \scrD(\Gamma)$ we can estimate
\begin{align}
\label{}
\norm{T_k \phi}{H^{s-1}_k(\Gamma)} \leq \norm{T_k^\infty \phi}{H^{s-1}_k(\R^{n-1})}
& = \frac{1}{2}\sqrt{\int_{\R^{n-1}}(k^2+|\bxi|^2)^{s-1} |Z(\bxi)|^2|\widehat{\phi}(\bxi)|^2 \, \rd \bxi}
\notag\\ &
 \leq \frac{1}{2}\sqrt{\int_{\R^{n-1}}(k^2+|\bxi|^2)^{s}| \widehat{\phi}(\bxi)|^2 \, \rd \bxi},%
\end{align}
as required.
\end{proof}

\begin{proof}[Proof of Theorem \ref{ThmTCoercive}]
We assume throughout that $L=\diam \Gamma = 1$, noting that a simple rescaling deals with the general case.
By the density of $\scrD(\Gamma)$ in $\tilde{H}^{1/2}(\Gamma)$ it suffices to prove \rf{bEst} for $\phi\in \scrD(\Gamma)$. For such a $\phi$, equation \rf{TIPFourierRep} from Theorem \ref{ThmFourierRep} gives that
\begin{align}
\label{bI}
|a_\sN(\phi,\phi)| = \frac{1}{2}\left|\int_{\R^{n-1}} Z(\bxi) |\widehat{\phi}(\bxi)|^2 \,\rd \bxi\right| \geq \frac{I}{2\sqrt{2}},
\end{align}
where
\begin{align*}
I:=\int_{\R^{n-1}} |Z(\bxi)| |\widehat{\phi}(\bxi)|^2 \,\rd \bxi.
\end{align*}
Defining
\begin{align*}
\label{}
J:=\norm{\phi}{\tilde{H}^{1/2}_k(\Gamma)}^2 =\int_{\R^{n-1}}(k^2+|\bxi|^2)^{1/2}|\widehat{\phi}(\bxi)|^2\,\rd \bxi,
\end{align*}
the problem of proving \rf{bEst} reduces to that of proving
\begin{align}
\label{JIIneq}
I \geq Ck^\beta J, \qquad k\geq k_0,
\end{align}
for some $C>0$ depending only on $k_0$. %

 The difficulty in proving \rf{JIIneq} is that the factor $|Z(\bxi)|$ in $I$ vanishes when $|\bxi|=k$. To deal with this, we write the integrals $I$ and $J$ as
\begin{align*}
\label{}
I&= I_1+I_2+I_3+I_4, \qquad\qquad
J= J_1+J_2+J_3+J_4,
\end{align*}
corresponding to the decomposition
\begin{align*}
\label{}
\int_{\R^{n-1}} = \int_{0<|\bxi|< k-\eps} + \int_{k-\eps<|\bxi|< k} + \int_{k<|\bxi|<k+\eps} + \int_{|\bxi|>k+\eps},
\end{align*}
where $0<\eps\leq k$ is to be specified later. We then proceed to estimate the integrals $J_1,...,J_4$ separately. Throughout the remainder of the proof $c>0$ denotes an absolute constant whose value may change from occurrence to occurrence.

We first observe that, for $0<|\bxi|<k-\eps$,
\begin{align*}
\label{}
\frac{k^2+|\bxi|^2}{k^2-|\bxi|^2}
\leq\frac{k^2+(k-\eps)^2}{k^2-(k-\eps)^2}
\leq \frac{2k^2}{\eps(2k-\epsilon)} \leq \frac{2k}{\epsilon},
\end{align*}
so that
\begin{align*}
J_1 :&= \int_{0<|\bxi|<k-\eps} (k^2+|\bxi|^2)^{1/2}|\widehat{\phi}(\bxi)|^2\,\rd \bxi \\
& = \int_{0<|\bxi|<k-\eps} \frac{(k^2+|\bxi|^2)^{1/2}}{ (k^2-|\bxi|^2)^{1/2}} |Z(\bxi)|\,|\widehat{\phi}(\bxi)|^2\,\rd \bxi \\
&\leq c \sqrt{\frac{k}{\eps}}\, I_1.
\end{align*}
Similarly, for $|\bxi|>k+\eps$,
\begin{align*}
\label{}
\frac{k^2+|\bxi|^2}{|\bxi|^2-k^2}
\leq \frac{k^2+(k+\eps)^2}{(k+\eps)^2-k^2} \leq \frac{5k}{2\eps},
\end{align*}
so that
\begin{align*}
J_4 :&= \int_{|\bxi|>k+\eps}(k^2+|\bxi|^2)^{1/2}|\widehat{\phi}(\bxi)|^2\,\rd \bxi\\
&
= \int_{|\bxi|>k+\eps} \frac{(k^2+|\bxi|^2)^{1/2}}{ (|\bxi|^2-k^2)^{1/2}} |Z(\bxi)|\,|\widehat{\phi}(\bxi)|^2\,\rd \bxi \\
&\leq c \sqrt{\frac{k}{\eps}} \,I_4.
\end{align*}

To estimate $J_2$ and $J_3$, we first derive a pointwise estimate on the Fourier transform of $\phi$. To do this, we note first that, for $t\in \R$, $|\re^{\ri t}-1|^2 = 4\sin^2{(t/2)}\leq t^2$, so that, for $\bxi_1,\bxi_2\in \R^{n-1}$ and recalling our assumption that $\diam \Gamma = 1$,
\begin{align*}
\label{}
|\widehat{\phi}(\bxi_1)-\widehat{\phi}(\bxi_2)| &\leq \frac{1}{(2\pi)^{\frac{n-1}{2}}}\left|\int_{ \Gamma}\re^{-\ri \bxi_2\cdot \bx }\left(\re^{-\ri (\bxi_1-\bxi_2)\cdot \bx  }- 1 \right)\phi(\bx )\,\rd \bx  \right|\nonumber \\
&\leq \frac{|\bxi_1-\bxi_2|}{(2\pi)^{\frac{n-1}{2}}}\int_{ \Gamma}|\bx ||\phi(\bx )|\,\rd \bx  \nonumber \\
&\leq \frac{|\bxi_1-\bxi_2|}{(2\pi)^{\frac{n-1}{2}}}\left(\int_{ \Gamma}|\bx |^2\,\rd \bx \right)^{1/2} \left(\int_{ \Gamma}|\phi(\bx )|^2\,\rd \bx \right)^{1/2}.  \nonumber \\
&=c |\bxi_1-\bxi_2|\left(\int_{\R^{n-1}}|\widehat{\phi}(\bxi)|^2\,\rd \bxi\right)^{1/2},  \nonumber \\
&\leq c|\bxi_1-\bxi_2|k^{-1/2}J.
\end{align*}
As a result, we can estimate, with $\hat{\bxi}:=\bxi/|\bxi|$,
\begin{align}
\label{FTPointwise}
|\widehat{\phi}(\bxi)|^2 \leq 2\left (|\widehat{\phi}(\bxi\pm\eps \hat{\bxi})|^2 + |\widehat{\phi}(\bxi\pm\eps \hat{\bxi}) -
\widehat{\phi}(\bxi)|^2\right) \leq 2\left (|\widehat{\phi}(\bxi\pm\eps \hat{\bxi})|^2 + \frac{c\eps^2  J}{k}\right), %
\end{align}
which %
then implies that
\begin{align}
J_2:&= \int_{k-\eps<|\bxi|<k} (k^2+|\bxi|^2)^{1/2}|\widehat{\phi}(\bxi)|^2\,\rd \bxi %
\notag \\
&
 \leq 2\sqrt{2}k\left(\int_{k-\eps<|\bxi|<k} |\widehat{\phi}(\bxi-\eps \hat{\bxi})|^2 \,\rd \bxi + c\eps^3 k^{n-3} J \right).
 \label{J2Est}
\end{align}
We now note that, for $0<\eps<c<d$,
\begin{align}
\label{fxipluseps}
\int_{c<|\bxi|<d} f(\bxi \pm \eps \hat{\bxi}) \,\rd \bxi = \begin{cases}
\int_{c\pm \eps<|\bxi|<d\pm \eps} f(\bxi)\left(1\mp \frac{\eps}{|\bxi|}\right) \,\rd \bxi,& n=3,\\
\int_{c\pm \eps<|\bxi|<d\pm \eps} f(\bxi) \,\rd \bxi,& n=2.
\end{cases}
\end{align}
Assume that $0<\eps<k/3$. Then for $k-2\eps<|\bxi|<k-\eps$, we have $1+\eps/|\bxi|\leq 2$, so that, using \rf{fxipluseps},
\begin{align*}
\label{}
\int_{k-\eps<|\bxi|<k} |\widehat{\phi}(\bxi-\eps \hat{\bxi})|^2 \,\rd \bxi &\leq 2 \int_{k-2\eps<|\bxi|<k-\eps} |\widehat{\phi}(\bxi )|^2 \,\rd \bxi\nonumber\\
&\leq \frac{2}{(k^2-(k-\eps)^2)^{1/2}}\int_{k-2\eps<|\bxi|<k-\eps} (k^2-|\bxi|^2)^{1/2}|\widehat{\phi}(\bxi )|^2 \,\rd \bxi\nonumber\\
&\leq \frac{2}{(\eps k)^{1/2}}I_1.
\end{align*}
Inserting this estimate into \rf{J2Est}, we find that
\begin{align*}
J_2 \leq c \sqrt{\frac{k}{\eps}} \,I_1 + c\eps^3 k^{n-2} J.
\end{align*}
Arguing similarly, again assuming that $0<\eps<k/3$, but using \rf{FTPointwise} and \rf{fxipluseps} with the plus rather than the minus sign, gives
\begin{align*}
\label{}
J_3 :&= \int_{k<|\bxi|<k+\eps} (k^2+|\bxi|^2)^{1/2}|\widehat{\phi}(\bxi)|^2\,\rd \bxi\leq c\sqrt{\frac{k}{\eps}}\,I_4 + c\eps^3 k^{n-2} J.
\end{align*}

Combining the above estimates we see that, for $0<\eps<k/3$,%
\begin{align*}
J = J_1+J_2+J_3+J_4 & \leq c\sqrt{\frac{k}{\eps}} \left(I_1 + I_4\right) +  c\eps^3 k^{n-2} J,
\end{align*}
which implies that
\begin{align}
\label{JEst2}
J\left(1-c\eps^3k^{n-2}\right) \leq c\sqrt{\frac{k}{\eps}}\,I.
\end{align}
Now, given $k_0>0$, choose $\tilde c>0$ such that $\tilde c< 1/(2c)$ and $\tilde c k^{-(n-2)/3} \leq k/3$, for $k\geq k_0$. Then, for $k\geq k_0$, setting $\eps = \tilde c k^{-(n-2)/3}$ in \rf{JEst2}, it follows from \rf{JEst2} that
$$
J \leq c \,\tilde c^{-1/2} \,k^{-\beta} I,
$$
where $\beta = -1/2 - (n-2)/6$.
Thus \rf{JIIneq} holds, which completes the proof.
\end{proof}

\begin{rem}
\label{TRem}
We can show that the bounds established in Theorem \ref{ThmTNormSmooth} are sharp in their dependence on $k$ as $k\to\infty$.  Let $0\neq \phi\in \scrD(\Gamma)$ be independent of $k$. Then, by \rf{bI}, and since $a_\sN(\phi,\phi) = (T_k\phi,\phi)_{L^2(\Gamma)}$ and $\widehat{\phi}(\bxi)$ is rapidly decreasing,%
\begin{align}
\label{bEstRem}
|(T_k\phi,\phi)_{L^2(\Gamma)}| \geq \frac{1}{2\sqrt{2}} \int_{\R^{n-1}}\sqrt{|k^2-|\bxi|^2|}|\widehat{\phi}(\bxi)|^2 \,\rd \bxi
& \sim  \frac{k}{2\sqrt{2}} \int_{\R^{n-1}}\,|\widehat{\phi}(\bxi)|^2 \,\rd \bxi,
\end{align}
as $k\to\infty$. Also, for every $s\in\R$, $|(T_k\phi,\phi)_{L^2(\Gamma)}| \leq  \norm{T_k\phi}{H^{s-1}_k(\Gamma)} \norm{\phi}{\tilde{H}^{1-s}_k(\Gamma)}$ and
\begin{align}
\label{bEstRem2}
 \norm{\phi}{\tilde{H}^{s}_k(\Gamma)}^2 %
= \int_{\R^{n-1}} (k^2+|\bxi|^2)^s |\widehat{\phi}(\bxi)|^2\,\rd \bxi \sim  k^{2s} \int_{\R^{n-1}} |\widehat{\phi}(\bxi)|^2\,\rd \bxi,
\end{align}
as $k\to\infty$.
Combining \rf{bEstRem} and \rf{bEstRem2} we see that, for every $s\in\R$ and $C<1/(2\sqrt{2})$,  it holds for all sufficiently large $k$ that $|(T_k\phi,\phi)_{L^2(\Gamma)}| \geq C\norm{\phi}{\tilde{H}^{1-s}_k(\Gamma)} \norm{\phi}{\tilde{H}^{s}_k(\Gamma)}$,  so that
\begin{align*}
\label{}
 \norm{T_k\phi}{H^{s-1}_k(\Gamma)}  \geq  C \norm{\phi}{\tilde{H}^{s}_k(\Gamma)},
\end{align*}
for all $k$ sufficiently large, which demonstrates the sharpness of \rf{TEst} in the limit $k\to\infty$.
\end{rem}

\begin{rem}\label{rem:sharpness7} %
We can also show that the bound established in Theorem \ref{ThmTCoercive} is sharp in its dependence on $k$ as $k\to\infty$, in the case $n=2$.  As in Remark \ref{rem:sharpness1}, let $\phi(\tbx):=\re^{\ri k\tbd\cdot \tbx}\psi(\tbx)$, where $\tbd\in \R^{n-1}$ is a unit vector and $0\neq \psi\in \scrD(\Gamma)$ is independent of $k$, so that
$\widehat{\phi}(\bxi) = \widehat{\psi}(\boldsymbol{\eta})$, where $\boldsymbol{\eta}=\bxi-k\tbd$. Since $|k^2-|\bxi|^2| \leq 2k |\boldsymbol{\eta}| + |\boldsymbol{\eta}|^2$, and since $\widehat{\phi}(\bxi)$ is rapidly decreasing,
\begin{align}
\label{CRem10}
|a_\sN(\phi,\phi)| \leq
\frac{1}{2 } \int_{\R^{n-1}} \sqrt{2k |\boldsymbol{\eta}| + |\boldsymbol{\eta}|^2}\,|\widehat{\psi}(\boldsymbol{\eta})|^2 \,\rd \boldsymbol{\eta} \sim
\frac{k^{1/2}}{\sqrt{2} } \int_{\R^{n-1}} |\boldsymbol{\eta}|^{1/2}\,|\widehat{\psi}(\boldsymbol{\eta})|^2 \,\rd \boldsymbol{\eta},
\end{align}
as $k\to\infty$. Further,
\begin{align}
\label{CRem20}
 \norm{\phi}{\tilde{H}^{1/2}_k(\Gamma)}^2 = \int_{\R^{n-1}} \left(2k^2 +2k\boldsymbol{\eta}\cdot\tbd+|\boldsymbol{\eta}|^2\right)^{1/2}\,|\widehat{\psi}(\boldsymbol{\eta})|^2 \,\rd \boldsymbol{\eta} \sim \sqrt{2}\, k \int_{\R^{n-1}} |\widehat{\psi}(\boldsymbol{\eta})|^2 \,\rd \boldsymbol{\eta},
\end{align}
as $k\to\infty$.
Combining \rf{CRem10} and \rf{CRem20} we see that, for some constant $C>0$ independent of $k$,
$$
|a_\sN(\phi,\phi)| \leq Ck^{-1/2}  \norm{\phi}{\tilde{H}^{1/2}_k(\Gamma)}^2,
$$
for all sufficiently large $k$.
This demonstrates the sharpness of \rf{bEst} in the limit $k\to\infty$, for the case $n=2$. In the case $n=3$ it may be that \rf{bEst} holds with the value of $\beta$ increased from $-2/3$ to $-1/2$, i.e., to its value for $n=2$.
\end{rem}

\section{Norm estimates in $H^{1/2}(\Gamma)$}
\label{sec:NormEstimates}
In this section we derive $k$-explicit estimates of the norms of certain functions in $H^{1/2}(\Gamma)$, which are of relevance to the numerical solution of the Dirichlet boundary value problem $\sD$, when it is solved via the integral equation formulation \rf{BIE_sl}. For an application of the results presented here see \cite{ScreenBEM}.

The motivation for the estimates we prove in Lemma \ref{lem:H_one_half_norms} below comes from the need to estimate integrals (strictly speaking, duality pairings) of the form
\begin{align}
\label{IDef}
I:=\int_\Gamma w(\by) v(\by)\,\rd s(\by)
:=  \langle w,\overline{v} \rangle_{H^{1/2}(\Gamma)\times \tilde{H}^{-1/2}(\Gamma)},
\end{align}
where $w\in H^{1/2}(\Gamma)$ and $v\in \tilde{H}^{-1/2}(\Gamma)$.

One situation in which such integrals arise is when solving \rf{BIE_sl} using a Galerkin BEM. In order to derive error estimates for the resulting solution in the domain $D$, and for the far-field pattern (defined e.g.\ as in \cite[Eqn (52)]{ScreenBEM}), we need to estimate duality pairings of the form \rf{IDef} where $v\in \tilde{H}^{-1/2}(\Gamma)$ represents the error in our Galerkin solution, and $w\in H^{1/2}(\Gamma)$ is a known function, possibly depending on a parameter. For the far-field pattern, $w(\by)= \re^{\ri k\hat{\bx}\cdot \by}$ for some observation direction $\hat{\bx}\in\R^n$ with $|\hat{\bx}|= 1$, and, for the solution evaluated at $\bx\in D$, $w(\by)= \Phi(\bx,\by)$.

One also encounters integrals of the form \rf{IDef} when attempting to estimate the magnitude of the solution of the continuous problem at a point $\bx$ in the domain, using a bound on the boundary data (an example is Corollary \ref{cor:BoundOnSolnInDomain} below, which is applied in \cite{ScreenBEM}). In this case $w(\by)=\Phi(\bx,\by)$ and $v=[\pdonetext{u}{\bn}]$ is the exact solution of \rf{BIE_sl}.

Given an estimate of $\normt{v}{\tilde{H}^{-1/2}_k(\Gamma)}$, an estimate of $|I|$ follows from
\begin{align*}
\label{}
|I| =| \langle w,\overline{v} \rangle_{H^{1/2}(\Gamma)\times \tilde{H}^{-1/2}(\Gamma)}| \leq  \normt{w}{H^{1/2}_k(\Gamma)}\normt{v}{\tilde{H}^{-1/2}_k(\Gamma)},
\end{align*}
provided we can bound $\normt{w}{H^{1/2}_k(\Gamma)}$. We now do this for the choices of $w$ noted above.

\begin{lem}
\label{lem:H_one_half_norms}
Let $k>0$, let $\Gamma$ be an arbitrary nonempty relatively open subset of $\Gamma_\infty$, and let $L:=\diam{\Gamma}$.
\begin{enumerate}[(i)]
\item Let $\bd\in\R^n$ with $|\bd|\leq 1$. Then, for $s\geq 0$, there exists $C_s>0$, dependent only on $s$, such that
\begin{align}
 \label{eqn:H_one_half_planewave}
 \normt{\re^{\ri k \bd\cdot (\cdot)}}{H^{s}_k(\Gamma)} \leq C_s L^{(n-1-2s)/2}(1+kL)^s.
 \end{align}
\item Let $\bx\in D:=\R^n\setminus\overline{\Gamma}$. Then there exists $C>0$, independent of $k$, $\Gamma$, and $\bx$, such that
 \begin{align}
 \label{eqn:H_one_half_Phi}
 \norm{\Phi(\bx,\cdot)}{H^{1/2}_k(\Gamma)} \leq
\begin{cases}
 \displaystyle{C\sqrt{k}\left(\frac{1}{kd}+P_3(L/d)\right)}, & n=3,\\[4mm]
  \displaystyle{C\left(\frac{1}{\sqrt{kd}}\left(\frac{1}{\sqrt{kL}} + \log\left(2+\frac{1}{kd}\right)\right) +\,P_2(L/d)\right)}, & n=2,
\end{cases}
 \end{align}
 where $d:=\dist(\bx,\Gamma)$ and $P_n(t) := \min(t^{(n-1)/2}, \log^{1/2}(2+t))$.
\end{enumerate}
\end{lem}

\begin{proof}
Choose $\chi\in C^\infty(\R)$ such that $\chi(t) = 0$, for $t\geq 2$, $\chi(t) = 1$, for $t\leq 1$. In both parts (i) and (ii) we wish to estimate $\normt{u}{H^{1/2}(\Gamma)}$, where $u\in H^{1/2}(\Gamma)$ is such that  $u=\tilde{u}|_\Gamma$ for some $\tilde{u}\in L^1_{\rm loc}(\R^{n-1})$ with $\tilde{u}|_{\overline \Gamma}\in \scrD(\overline \Gamma)$.

 Consider first part (i), in which, for some $\bd\in \R^n$ with $|\bd|\leq 1$, $\tilde{u}(\tilde \by)= \re^{\ri k\bd\cdot \by}$, for $\tilde \by\in \R^{n-1}$, where $\by = (\tilde \by, 0)\in \Gamma_\infty$. Suppose without loss of generality that the origin lies within $\Gamma$, and define $\chi_L\in \scrD(\R^{n-1})$ by $\chi_L(\tilde \by) = \chi(|\tilde \by|/L)$, for $\tilde \by\in \R^{n-1}$. Then, for any $s\in \R$, since $(\chi_L \tilde u)|_\Gamma = u$,
\begin{align}
\label{uChi}
\normt{u}{H^{s}_k(\Gamma)} \leq \normt{\chi_L \tilde{u}}{H^{s}_k(\R^{n-1})}.
\end{align}
Moreover, the standard shift and scaling theorems for the Fourier transform imply that $\widehat{\chi_L \tilde{u}}(\bxi) = L^{n-1}\widehat{\chi}(L(\bxi-k\bd))$. Thus, for $s\geq 0$, substituting $\boldeta=L(\bxi-k\bd)$,
\begin{align*}
\label{}
\normt{\chi_L \tilde{u}}{H^{s}_k(\R^{n-1})}^2 &= L^{2(n-1)}\int_{\R^{n-1}} (k^2+|\bxi|^2)^{s}|\widehat{\chi}(L(\bxi-k\bd))|^2 \, \rd\bxi\notag\\
&= L^{n-1-2s}\int_{\R^{n-1}} ((kL)^2+|\boldeta+kL\bd|^2)^{s}|\widehat{\chi}(\boldeta)|^2 \, \rd\boldeta\notag\\
&\leq L^{n-1-2s}(1+kL)^{2s}\int_{\R^{n-1}} (2+|\boldeta|)^{2s}|\widehat{\chi}(\boldeta)|^2 \, \rd\boldeta\notag,
\end{align*}
since $2(kL)^2+2|\boldeta|kL + |\boldeta|^2 \leq(1+kL)^{2}(2+|\boldeta|)^{2}$.

For part (ii), where, for some $\bx\in D$, $\tilde{u}(\tilde \by)= \Phi(\bx,(\tilde \by,0))$, for $\tilde \by\in \R^{n-1}$, and $u=\tilde u|_\Gamma$,
such a direct approach is not possible. Instead, noting that $u\in H^1(\Gamma)$, we will bound $\|u\|_{H^{1/2}_k(\Gamma)}$ using the simple estimate that
\begin{align}
\label{embed}
\|u\|_{H^{1/2}_k(\Gamma)}\leq k^{-1/2} \|u\|_{H^{1}_k(\Gamma)}.
\end{align}
Now, for $\tilde \by\in\R^{n-1}$, $\tilde u(\tilde \by) = k^{n-2}F(kr(\tilde \by))$, where $r(\tilde \by) = \sqrt{|\tilde \by-\tilde \bx|^2+x_n^2}$ and, for $t>0$, $F(t):= \frac{\ri}{4}H_0^{(1)}(t)$ when $n=2$, while $F(t) := \re^{\ri t}/(4\pi t)$ when $n=3$. Further, $|\nabla \tilde u(\tilde \by)|\leq k^{n-1} |F^\prime(kr(\tilde\by))|$. Recalling \rf{H0Est}-\rf{B0Der}, we see that
\begin{align}
\label{Fbounds}
|F(t)| \leq C t^{-(n-1)/2}, \quad |F^\prime(t)| \leq C \left(t^{-(n-1)/2}+t^{-(n-1)}\right), \quad \mbox{for } t>0,
\end{align}
where $C>0$ is an absolute constant, not necessarily the same at each occurrence.

Suppose first that $d=\dist(\bx, \Gamma)\geq 4L$. Then, where $\chi_L$ is defined as above,
\begin{align} \label{b1}
\|u\|^2_{H^{1}_k(\Gamma)} \leq \|\chi_L \tilde u\|^2_{H^{1}_k(\R^{n-1})} = k^2 \|\chi_L \tilde u\|^2_{L^2(\R^{n-1})}+\|\nabla(\chi_L \tilde u)\|^2_{L^2(\R^{n-1})}.
\end{align}
Note that $r(\tilde \by) \geq d-2L\geq d/2$ and $|\nabla \chi_L(\by)|\leq C/L$ for $\tilde \by\in \supp \chi_L$,  and that $\supp \chi_L$ has measure $\leq CL^{n-1}$. Thus, and using \rf{Fbounds} and \rf{b1}, and recalling that $|F(z)|$ is monotonically decreasing on $z>0$, we see that
\begin{align} \label{finalfar}
\|u\|_{H^{1}_k(\Gamma)} \leq C L^{(n-1)/2} \left( L^{-1}k^{n-2}|F(kd/2)| + (k/d)^{(n-1)/2} + d^{1-n}\right).
\end{align}

Now consider the case that $d< 4L$.  Write $\bx$ as $\bx=(\tilde \bx,x_n)$ with $\tilde\bx\in \R^{n-1}$, and set $S_R:=\{\tilde \by\in \R^{n-1}:|\tilde \bx-\tilde \by|<R\}$, for $R>0$. Modify the definition of $\chi_L$, setting $\chi_L(\tilde \by) = \chi(|\tilde \by|/L)(1-\chi(2 r(\tilde\by)/d))$. With this definition it still holds that $\chi_L=1$ on the closure of $\Gamma$, so $(\chi_L\tilde u)|_\Gamma = u$. Also  $|\nabla \chi_L(\tilde \by)|\leq C/d$; in fact $|\nabla \chi_L(\tilde \by)|\leq C/L$, for $|\tilde \by-\tilde \bx|\geq d$. Again \rf{b1} holds.  Recalling that $|\nabla \tilde u(\by)|\leq k^{n-1} |F^\prime(kr(\tilde\by))|$, and using \rf{Fbounds} and that $r(\tilde \by)\geq |\tilde \by-\tilde \bx|$, and $r(\tilde \by)\geq d/2$ on the support of $\chi_L$, we see that
\begin{align} \notag
\|\chi_L\nabla \tilde u\|^2_{L^2(\R^{n-1})} &\leq k^{2n-2}\left(\int_{S_d}\chi^2_L(\tilde \by)|F^\prime(kr(\tilde\by))|^2 \rd \tilde \by +C\int_{S_{5L}\setminus S_{d}}|F^\prime(kr(\tilde\by))|^2 \rd \tilde \by\right) \\ \notag
& \leq Ck^{n-1} + C d^{1-n} + C \int_{d}^{5L}\left(\frac{k^{n-1}}{t}+\frac{1}{t^n}\right) \rd t\\ \label{D1}
& \leq C k^{n-1} \log(5L/d) + C d^{1-n}.
\end{align}
Similarly,
\begin{align} \notag
k^2\|\chi_L\tilde u\|^2_{L^2(\R^{n-1})} &\leq k^{2n-2}\left(\int_{S_d}\chi^2_L(\tilde \by)|F(kr(\tilde\by))|^2 \rd \tilde \by + C\int_{S_{5L}\setminus S_d}|F(kr(\tilde\by))|^2 \rd \tilde \by\right)\\  \label{D2}
& \leq C k^{n-1}\left(1 + \int_d^{5L}t^{-1}\,\rd t\right) \leq C k^{n-1} \log(5L/d)
\end{align}
 and, using that $|F(z)|$ is monotonic for $z>0$,
\begin{align} \notag
\|\tilde u\nabla\chi_L\|^2_{L^2(\R^{n-1})} &\leq Ck^{2n-4}\left(d^{n-3}|F(kd/2)|^2 + L^{-2}\int_{S_{5L}\setminus S_d}|F(kr(\tilde\by))|^2 \rd \tilde \by\right)\\ \label{D3}
&\leq C k^{2n-4}d^{n-3}|F(kd/2)|^2 + C k^{n-3} L^{-2}\log(5L/d).
\end{align}
Combining \rf{b1} and \rf{D1}--\rf{D3}, we see that, for $d<4L$,
\begin{align} 
\|u\|_{H^{1}_k(\Gamma)} & \leq C k^{(n-3)/2}(k+L^{-1})\log^{1/2}(5L/d) \notag \\
& \qquad \qquad + C d^{(1-n)/2} + C k^{n-2}d^{(n-3)/2}|F(kd/2)|.
\label{D4}
\end{align}

Now, in the case $n=3$, for which $|F(kd/2)|\leq C/(kd)$, it follows from \rf{finalfar} and \rf{D4}, and noting that $\log^{1/2}(5L/d) \leq CP_3(L/d) \leq CL/d$ for $4L>d$, that
\begin{align} \label{F1}
\|u\|_{H^{1}_k(\Gamma)} \leq \frac{C}{d}\Big(1+kd\,P_3(L/d)\Big).
\end{align}
 For $n=2$, $F(kd/2) \leq C\log(2+(kd)^{-1})(1+kd)^{-1/2}\leq C(kd)^{-1/2}$,  by \rf{H0Est} and \rf{BnEst}. Hence, and by \rf{finalfar} and \rf{D4} and as $\log^{1/2}(5L/d) \leq CP_2(L/d) \leq CL/d$ for $4L>d$,
\begin{align} \label{F2}
\|u\|_{H^{1}_k(\Gamma)} \leq \frac{C}{d^{1/2}}\Big((kL)^{-1/2} + \log(2+(kd)^{-1}) +(kd)^{1/2}\,P_2(L/d)\Big).
\end{align}
Part (ii) of the lemma then follows from \rf{embed}, \rf{F1}, and \rf{F2}.
\end{proof}

As an application we use Lemma \ref{lem:H_one_half_norms} to prove a $k$-explicit pointwise bound on the solution of the sound-soft screen scattering problem considered in Example \ref{ex:scattering}.
\begin{cor}
\label{cor:BoundOnSolnInDomain}
The solution $u$ of problem $\sD$, with $g_{\sD}=-u^i|_\Gamma$, satisfies the pointwise bound
\begin{align*}
\label{}
|u(\bx)|\leq
\begin{cases}
 \displaystyle{C\sqrt{kL}\sqrt{1+kL}\left(\frac{1}{kd}+P_3(L/d)\right)}, & n=3,\\[3mm]
\displaystyle{C\sqrt{1+kL}\left(\frac{1}{\sqrt{kd}}\left(\frac{1}{\sqrt{kL}} + \log\left(2+\frac{1}{kd}\right)\right) +P_2(L/d)\right)}, & n=2,
\end{cases}
\end{align*}
where $\bx\in D$, $d:=\dist(\bx,\Gamma)$, $L:=\diam{\Gamma}$, and $C>0$ is independent of $k$, $\Gamma$ and $\bx$.
\end{cor}
\begin{proof}
Using Theorem \ref{DirEquivThm} we can estimate
\begin{align*}
\label{}
|u(\bx)|
&= \left| \cS_k \left[\pdonetext{u}{\bn}\right](\bx)\right| \\
&= \left|\langle \Phi(\bx,\cdot),\overline{\left[\pdonetext{u}{\bn}\right]} \rangle_{H^{1/2}(\Gamma)\times \tilde{H}^{-1/2}(\Gamma)}\right|\\
&\leq \norm{\Phi(\bx,\cdot)}{H^{1/2}(\Gamma)}\norm{S_k^{-1}}{H^{1/2}(\Gamma)\to\tilde{H}^{-1/2}(\Gamma)}\norm{u^i|_\Gamma}{H^{1/2}(\Gamma)},
\end{align*}
and the result follows from applying Lemma \ref{lem:H_one_half_norms} to estimate the first and third factors, and using Theorem \ref{ThmCoerc} (and the Lax Milgram lemma) which give the bound $\norm{S_k^{-1}}{H^{1/2}(\Gamma)\to\tilde{H}^{-1/2}(\Gamma)}\leq 2\sqrt{2}$.
\end{proof}

\section{Acknowledgements}
This work was supported by EPSRC grant EP/F067798/1. The authors are grateful to A.\ Moiola for many stimulating discussions in relation to this work.
\bibliography{BEMbib_short2014}%
\bibliographystyle{siam}
\end{document}

%% file: macros.tex
\newcommand{\done}[2]{\dfrac{d {#1}}{d {#2}}}
\newcommand{\donet}[2]{\frac{d {#1}}{d {#2}}}
\newcommand{\pdone}[2]{\dfrac{\partial {#1}}{\partial {#2}}}
\newcommand{\pdonet}[2]{\frac{\partial {#1}}{\partial {#2}}}
\newcommand{\pdonetext}[2]{\partial {#1}/\partial {#2}}
\newcommand{\pdtwo}[2]{\dfrac{\partial^2 {#1}}{\partial {#2}^2}}
\newcommand{\pdtwot}[2]{\frac{\partial^2 {#1}}{\partial {#2}^2}}
\newcommand{\pdtwomix}[3]{\dfrac{\partial^2 {#1}}{\partial {#2}\partial {#3}}}
\newcommand{\pdtwomixt}[3]{\frac{\partial^2 {#1}}{\partial {#2}\partial {#3}}}
\newcommand{\bs}[1]{\mathbf{#1}}
\newcommand{\bx}{\mathbf{x}}
\newcommand{\by}{\mathbf{y}}
\newcommand{\bd}{\mathbf{d}} 
\newcommand{\bn}{\mathbf{n}} 
\newcommand{\bP}{\mathbf{P}} 
\newcommand{\bp}{\mathbf{p}} 
\newcommand{\ol}[1]{\overline{#1}}
\newcommand{\rf}[1]{(\ref{#1})}
\newcommand{\xt}{\mathbf{x},t}
\newcommand{\hs}[1]{\hspace{#1mm}}
\newcommand{\vs}[1]{\vspace{#1mm}}
\newcommand{\eps}{\varepsilon}
\newcommand{\ord}[1]{\mathcal{O}\left(#1\right)} 
\newcommand{\oord}[1]{o\left(#1\right)}
\newcommand{\Ord}[1]{\Theta\left(#1\right)}
\newcommand{\PhiF}{\Phi_{\rm freq}}
\newcommand{\real}[1]{{\rm Re}\left[#1\right]} 
\newcommand{\im}[1]{{\rm Im}\left[#1\right]}
\newcommand{\hsnorm}[1]{||#1||_{H^{s}(\bs{R})}}
\newcommand{\hnorm}[1]{||#1||_{\tilde{H}^{-1/2}((0,1))}}
\newcommand{\norm}[2]{\left\|#1\right\|_{#2}}
\newcommand{\normt}[2]{\|#1\|_{#2}}
\newcommand{\on}[1]{\Vert{#1} \Vert_{1}}
\newcommand{\tn}[1]{\Vert{#1} \Vert_{2}}
\newcommand{\ts}{\tilde{s}}
\newcommand{\tGamma}{{\tilde{\Gamma}}}
\newcommand{\darg}[1]{\left|{\rm arg}\left[ #1 \right]\right|}
\newcommand{\bnabla}{\boldsymbol{\nabla}}
\newcommand{\dive}{\boldsymbol{\nabla}\cdot}
\newcommand{\curl}{\boldsymbol{\nabla}\times}
\newcommand{\Phixy}{\Phi(\bx,\by)}
\newcommand{\PhiOxy}{\Phi_0(\bx,\by)}
\newcommand{\dxPhixy}{\pdone{\Phi}{n(\bx)}(\bx,\by)}
\newcommand{\dyPhixy}{\pdone{\Phi}{n(\by)}(\bx,\by)}
\newcommand{\dxPhiOxy}{\pdone{\Phi_0}{n(\bx)}(\bx,\by)}
\newcommand{\dyPhiOxy}{\pdone{\Phi_0}{n(\by)}(\bx,\by)}

\newcommand{\rd}{\mathrm{d}}
\newcommand{\R}{\mathbb{R}}
\newcommand{\N}{\mathbb{N}}
\newcommand{\Z}{\mathbb{Z}}
\newcommand{\C}{\mathbb{C}}
\newcommand{\K}{{\mathbb{K}}}
\newcommand{\ri}{{\mathrm{i}}}
\newcommand{\re}{{\mathrm{e}}} 

\newcommand{\cA}{\mathcal{A}}
\newcommand{\cC}{\mathcal{C}}
\newcommand{\cS}{\mathcal{S}}
\newcommand{\cD}{\mathcal{D}}
\newcommand{\cone}{{c_{j}^\pm}}
\newcommand{\ctwo}{{c_{2,j}^\pm}}
\newcommand{\cthree}{{c_{3,j}^\pm}}

\newtheorem{thm}{Theorem}[section]
\newtheorem{lem}[thm]{Lemma}
\newtheorem{defn}[thm]{Definition}
\newtheorem{prop}[thm]{Proposition}
\newtheorem{cor}[thm]{Corollary}
\newtheorem{rem}[thm]{Remark}
\newtheorem{conj}[thm]{Conjecture}
\newtheorem{ass}[thm]{Assumption}